\documentclass[sn-mathphys,Numbered]{sn-jnl}

\usepackage{amsbsy,amsgen,amscd,bm,xspace}
\usepackage{arydshln}

\usepackage{stmaryrd}
\usepackage[mathscr]{euscript}

\usepackage{colonequals}
\newcommand{\defeq}{\colonequals}
\newcommand{\eqdef}{\equalscolon}

\usepackage{mathdots} % for iddots

\newcommand{\bC}{\mathbf{C}}
\newcommand{\bH}{\mathbf{H}}
\newcommand{\bO}{\mathbf{O}}

\newcommand{\R}{\mathbb{R}}

\newcommand{\prox}{\textnormal{prox}}
\newcommand{\X}{\mathcal{V}}

\newcommand{\bmat}[1]{\begin{bmatrix}#1\end{bmatrix}}
\newcommand{\pmat}[1]{\begin{pmatrix}#1\end{pmatrix}}
\newcommand{\sbmat}[1]{\left[\begin{smallmatrix}#1\end{smallmatrix}\right]}

\newcommand{\titleparagraph}[1]{\paragraph{#1}}

\usepackage{bookmark}

\newcommand{\algcomp}{\diamond}

\usepackage{algorithm}
\usepackage{algorithmicx}
\usepackage{algpseudocode}
\renewcommand{\Comment}[2][.5\linewidth]{%
  \leavevmode\hfill\makebox[#1][l]{\(\triangleright\)~#2}}

\usepackage{caption}
\usepackage{etoolbox}\AtBeginEnvironment{algorithmic}{\small}

\algnewcommand{\Initialize}[1]{%
  \State \textbf{Initialize:}
  \Statex \hspace*{\algorithmicindent}\parbox[t]{.8\linewidth}{\raggedright #1}
}

\usepackage{tikz}
\usepackage{tikz-cd}
\usetikzlibrary{shapes,positioning,calc}

\definecolor{dkred}{rgb}{.5,0,0}
\definecolor{dark-gray}{gray}{0.3}
\definecolor{dkgray}{rgb}{.4,.4,.4}
\definecolor{gray}{rgb}{.4,.4,.4}
\definecolor{grey}{rgb}{.4,.4,.4}
\definecolor{dkblue}{rgb}{0,0,.5}
\definecolor{medblue}{rgb}{0,0,.75}
\definecolor{rust}{rgb}{0.5,0.1,0.1}

\makeatletter
\newcommand{\pushright}[1]{\ifmeasuring@#1\else\omit\hfill$\displaystyle#1$\fi\ignorespaces}
\newcommand{\pushleft}[1]{\ifmeasuring@#1\else\omit$\displaystyle#1$\hfill\fi\ignorespaces}
\makeatother

\newcommand{\beas}{\begin{eqnarray*}}
\newcommand{\eeas}{\end{eqnarray*}}
\newcommand{\bea}{\begin{eqnarray}}
\newcommand{\eea}{\end{eqnarray}}
\newcommand{\beq}{\begin{equation}}
\newcommand{\eeq}{\end{equation}}
\newcommand{\bit}{\begin{itemize}}
\newcommand{\eit}{\end{itemize}}
\newcommand{\ben}{\begin{enumerate}}
\newcommand{\een}{\end{enumerate}}

\newcommand{\ba}{\begin{array}}
\newcommand{\ea}{\end{array}}
\newcommand{\bbm}{\begin{bmatrix}}
\newcommand{\ebm}{\end{bmatrix}}

\newcommand{\eg}{e.g.}
\newcommand{\ie}{i.e.}

\newcommand{\rank}{\mathop{\bf rank}}

\newcommand{\diag}{\mathop{\bf diag}}

\newcommand{\Msim}{\overset{\hat M}{\sim}}
\newcommand{\Mpsim}{\overset{\hat M'}{\sim}}

\newcommand{\norm}[1]{\left\|#1\right\|}

\newcommand{\argmin}{\mathop{\rm argmin}}

\newcommand{\E}[1]{\mathbb{E}}

\newcommand{\lin}{\textsc{Linnaeus}}
\renewcommand{\argmin}{\mathop{\textrm{argmin}}}

\usepackage{bbm}
\usepackage{hyperref}

\usepackage{graphicx}
\usepackage{amsmath,amssymb,amsfonts}
\usepackage{amsthm}
\usepackage{mathrsfs}
\usepackage[title]{appendix}
\usepackage{xcolor}
\usepackage{textcomp}
\usepackage{manyfoot}
\usepackage{booktabs}
\usepackage{comment}

\usepackage{listings}%
\definecolor{mykeyword}{rgb}{0.6, 0.2, 0.8}
\definecolor{mycomment}{rgb}{0.0, 0.5, 0.0}
\definecolor{mystring}{rgb}{0.8, 0.0, 0.0}
\definecolor{mybackground}{rgb}{0.95, 0.95, 0.92}

\lstdefinestyle{mypython}{
    backgroundcolor=\color{mybackground},
    commentstyle=\color{mycomment}\itshape,
    keywordstyle=\color{mykeyword}\bfseries,
    stringstyle=\color{mystring},
    basicstyle=\ttfamily\footnotesize,
    frame=single,
    breaklines=true,
    showstringspaces=false,
    tabsize=4,
    language=Python,
	upquote=true,
    xleftmargin=0.05\textwidth,
    xrightmargin=0.05\textwidth,
    framexleftmargin=5pt,
    framexrightmargin=5pt,
    framextopmargin=5pt,
    framexbottommargin=5pt
}

\usepackage[bottom]{footmisc} % footnotes at bottom of page.

\usepackage[capitalise]{cleveref}

\makeatletter
\renewcommand{\ALG@name}{Algo.} % change name of "Algorithm" to be shorter
\makeatother

\counterwithin{algorithm}{section}

\AtBeginEnvironment{appendices}{\crefalias{section}{appendix}\crefalias{subsection}{appendix}}

\usepackage{aliascnt} 
\newtheorem{theorem}{Theorem}
\newaliascnt{proposition}{theorem}
\newaliascnt{lemma}{theorem}
\newaliascnt{corollary}{theorem}
\newaliascnt{definition}{theorem}
\newaliascnt{remark}{theorem}
\newaliascnt{assumption}{theorem}

\newtheorem{proposition}[proposition]{Proposition}% 
\newtheorem{lemma}[lemma]{Lemma}% 
\newtheorem{corollary}[corollary]{Corollary}%

\aliascntresetthe{proposition}
\aliascntresetthe{lemma}
\aliascntresetthe{corollary}

\crefname{proposition}{Proposition}{Propositions}
\Crefname{proposition}{Proposition}{Propositions}
\crefname{lemma}{Lemma}{Lemmas}
\Crefname{lemma}{Lemma}{Lemmas}
\crefname{corollary}{Corollary}{Corollaries}
\Crefname{corollary}{Corollary}{Corollaries}

\theoremstyle{definition}
\newtheorem{remark}[remark]{Remark}
\newtheorem{definition}[definition]{Definition}

\aliascntresetthe{remark}
\aliascntresetthe{definition}
\aliascntresetthe{assumption}

\crefname{remark}{Remark}{Remarks}
\Crefname{remark}{Remark}{Remarks}
\crefname{definition}{Definition}{Definitions}
\Crefname{definition}{Definition}{Definitions}
\crefname{assumption}{Assumption}{Assumptions}
\Crefname{assumption}{Assumption}{Assumptions}

\begin{document}

\title[Article Title]{Algebraic characterization of equivalence between oracle-based iterative algorithms}

\author*[1]{\fnm{Laurent} \sur{Lessard}}\email{l.lessard@northeastern.edu}

\author[2]{\fnm{Madeleine} \sur{Udell}}\email{udell@stanford.edu}

\affil[1]{Northeastern University}

\affil[2]{Stanford University}

\abstract{
	When are two algorithms the same?
	How can we be sure a recently proposed algorithm is novel,
	and not a minor variation on an existing method?
	In this paper, we present a framework for reasoning about equivalence
	between a broad class of iterative algorithms,
	with a focus on algorithms designed for convex optimization.
	We propose several notions of what it means for two algorithms to be equivalent,
	and provide computationally tractable means to detect equivalence.
	Our main definition, oracle equivalence, states that two algorithms
	are equivalent if they result in the same sequence of calls to the function oracles
	(for suitable initialization).
	Borrowing from control theory,
	we use state-space realizations to represent algorithms
	and characterize algorithm equivalence via transfer functions.
	Our framework can also identify and characterize equivalence between
	algorithms that use different oracles that are related via a
	linear fractional transformation.
	Prominent examples include linear transformations and function conjugation.
	To support the paper, we have developed a software package named \lin{}
	that implements the framework to identify other
	iterative algorithms that are equivalent to an input algorithm.
}


\maketitle

%%%%%%%%%%%%%%%%%%%%%%%%%%%%%%%%%%%%%%%%%%%%%%%%%%%%%%%%%%%%%%%%%%%%%%%%%%%%%%%%%%%%%%%%%%%
\section{Introduction}\label{intro}

Large-scale optimization problems in machine learning, signal processing, and imaging have fueled ongoing interest in iterative optimization algorithms.
New optimization algorithms are regularly proposed to
capture more complicated models, reduce computational burdens,
or obtain stronger performance and convergence guarantees.

However, the \emph{novelty} of an algorithm can be difficult to establish
because algorithms can be written in different equivalent forms.
For example, \cref{algo_i1} was originally proposed by Popov \citep{popov1980modification}
in the context of solving saddle point problems.
This method was later generalized by Chiang et al.~\cite[\S4.1]{chiang2012online}
in the context of online optimization.
\Cref{algo_i2} is a reformulation of \cref{algo_i1}
adapted for use in generative adversarial networks (GANs)~\cite{gidel2018a}.
\Cref{algo_i3} is an adaptation of \emph{Optimistic Mirror Descent}~\cite{OMD_rakhlin}
used by Daskalakis et al.~\cite{daskalakis2018training} and also used to train GANs.
Finally, \cref{algo_i4} was proposed by Malitsky~\cite{malitsky2015projected}
for solving monotone variational inequality problems.
In all four algorithms, the vectors $x^k_1$ and $x^k_2$ are algorithm states,
$\eta$ is a tunable parameter,
and $F$ is the gradient of the loss function.

\vspace{-1em}
\noindent\hfil
\begin{minipage}[t]{0.47\textwidth}
	\begin{algorithm}[H]
		\centering
		\captionsetup{font=scriptsize}
		\caption{(Modified Arrow--Hurwicz)}
		\label{algo_i1}
		\scriptsize
		\begin{algorithmic}
			\For{$k=0,1, 2,\dots$}
			\State{$x^{k+1}_1 = x^k_1 - \eta F(x^k_2)$}
			\State{$x^{k+1}_2 = x^{k+1}_1 - \eta F(x^k_2)$}
			\EndFor
		\end{algorithmic}
	\end{algorithm}
\end{minipage}
\hfil
\begin{minipage}[t]{0.47\textwidth}
	\begin{algorithm}[H]
		\centering
		\captionsetup{font=scriptsize}
		\caption{(Extrapolation from the past)}
		\label{algo_i2}
		\scriptsize
		\begin{algorithmic}
			\For{$k=0,1, 2,\dots$}
			\State{$x^k_2 = x^k_1 - \eta F(x^{k-1}_2)$}
			\State{$x^{k+1}_1 = x^k_1 - \eta F(x^{k}_2)$}
			\EndFor
		\end{algorithmic}
	\end{algorithm}
\end{minipage}
\hfil

\noindent\hfill
\begin{minipage}[t]{0.47\textwidth}
	\begin{algorithm}[H]
		\centering
		\captionsetup{font=scriptsize}
		\caption{(Optimistic Mirror Descent)}
		\label{algo_i3}
		\scriptsize
		\begin{algorithmic}
		\For{$k=0,1, 2,\dots$}
		\State{$x^{k+1}_2 =  x^k_2 - 2\eta F(x^k_2) + \eta F(x^{k-1}_2)$}
		\EndFor
		\end{algorithmic}
	\end{algorithm}
\end{minipage}
\hfill
\begin{minipage}[t]{0.47\textwidth}
	\begin{algorithm}[H]
		\centering
		\captionsetup{font=scriptsize}
		\caption{(Reflected Gradient Method)}
		\label{algo_i4}
		\scriptsize
		\begin{algorithmic}
		\For{$k=0,1, 2,\dots$}
		\State{$x^{k+1}_1 =  x^k_1 - \eta F( 2x^k_1 - x^{k-1}_1 )$}
		\EndFor
		\end{algorithmic}
	\end{algorithm}
\end{minipage}
\hfill
\vspace{1em}

\Crefrange{algo_i1}{algo_i4} are equivalent in the sense that when suitably initialized,
the sequences $(x^k_1)_{k\ge 0}$ and $(x^k_2)_{k\ge 0}$ are identical for all four algorithms.\footnote{
In their original formulations, \cref{algo_i1,algo_i2,algo_i4}
included projections onto convex constraint sets.
We assume an unconstrained setting here for illustrative purposes.
Some of the equivalences no longer hold in the constrained case.}
Although these particular equivalences are not difficult to verify
and many have been explicitly pointed out in the literature,
for example in~\cite{gidel2018a}, algorithm equivalence is not always immediately apparent.

In this paper, we present a framework for reasoning about algorithm equivalence,
with the ultimate goal of making the analysis and design of algorithms more principled and streamlined.
This includes:
\begin{itemize}
	\item A universal way of representing algorithms, inspired by methods from control theory.
	\item Sensible definitions of what it means for algorithms to be equivalent.
	\item A computationally efficient way to verify whether two algorithms are equivalent.
	\item A software package implementing this framework named \lin{}.
The software takes an algorithm described using natural syntax as input,
and returns a canonical form with known names
and pointers to relevant literature.
\end{itemize}

This paper studies equivalence of oracle-based algorithms at the oracle interface; any underlying optimization problem is external to the framework.

Briefly, our method is to parse each algorithm to a standard form
as a linear system in feedback with a nonlinearity;
to compute the \emph{transfer function} of each linear system;
and to check whether certain key relationships hold between the transfer functions of the algorithms in question.

This paper is organized as follows.
In \cref{relatedwork}, we briefly summarize existing literature related to our work.
In \cref{example}, we introduce four examples of equivalent algorithms that motivate our framework.
In \cref{preliminary}, we briefly review important background on linear systems and optimization
used throughout the paper and in \cref{control}, we present our control-inspired mathematical framework for algorithm representation.
We formally define three notions of algorithm equivalence:
\emph{oracle equivalence} (\cref{oracle-equ}), \emph{shift equivalence} (\cref{shift-equ}), and \emph{LFT equivalence} (\cref{lft-equ})
to handle cases with: one oracle, multiple oracles, and different but related oracles, respectively.
We discuss further generalizations and applications in \cref{discussion} and conclude in \cref{conclusion}.

%%%%%%%%%%%%%%%%%%%%%%%%%%%%%%%%%%%%%%%%%%%%%%%%%%%%%%%%%%%%%%%%%%%%%%%%%%%%%%%%%%%%%%%%%%%
\section{Related work}\label{relatedwork}

Within the optimization literature, several standard forms have been proposed to represent
problems and algorithms.
For example, the CVX* modeling languages represent (disciplined) convex optimization problems in a
standard conic form,
building up the representations of complicated problems
from a few basic functions and a small set of composition rules \cite{cvx, gb08, udell2014convex, diamond2016cvxpy, shen2017disciplined}.
This paper builds on a foundation developed by Lessard et al. \cite{doi:10.1137/15M1009597}
that represents first-order algorithms as linear systems in feedback with a nonlinearity.
Lessard et al.\ use this representation to analyze convergence properties of an algorithm with integral quadratic constraints.
Our work generalizes their representation to algorithms that use multiple related oracles.

There are rich connections between many first-order methods for convex optimization.
These algorithms are surveyed in a recent textbook by Ryu and Yin,
which summarizes and unifies several operator splitting methods for convex optimization \cite{ryuyinconvex}.
Many of these connections are well known to experts, but the connections have
traditionally been complicated to explain, communicate, or even remember.
For example, Boyd et al.~\cite{MAL-016} write,
``There are also a number of other algorithms distinct from but inspired
by ADMM. For instance, Fukushima \cite{applicationadmm} applies ADMM to a dual
problem formulation, yielding a `dual ADMM' algorithm, which is
shown in \cite{reformulationadmm} to be equivalent to the `primal Douglas--Rachford' method
discussed in \cite[\S3.5.6]{phdthesis}.''
As another example, Chambolle and Pock in~\cite{chambolle2011first}
proposed a new primal-dual splitting algorithm and demonstrated
that transformations of their algorithm can yield
Douglas--Rachford splitting and ADMM, using a full page of mathematics
to sketch the connection.
Using our framework, the relations between
the Chambolle--Pock method, Douglas--Rachford splitting, and ADMM
can be established precisely and automatically.

%%%%%%%%%%%%%%%%%%%%%%%%%%%%%%%%%%%%%%%%%%%%%%%%%%%%%%%%%%%%%%%%%%%%%%%%%%%%%%%%%%%%%%%%%%%
\section{Motivating examples}\label{example}

\crefrange{algo_i1}{algo_i4} discussed in \cref{intro} were equivalent in a strong sense; the iterates were in exact correspondence. In this paper, we adopt a broader view of equivalence, which we now illustrate with four motivating examples. Each example provides a different way that we consider two algorithms to be equivalent.

\noindent
\hfil
\begin{minipage}{0.47\textwidth}
	\begin{algorithm}[H]
		\centering
		\captionsetup{font=scriptsize}
		\caption{}
		\label{algo1}
		\scriptsize
		\begin{algorithmic}
			\For{$k=0, 1, 2,\ldots$}
			\State{$x^{k+1}_1 = 2x^k_1 - x^k_2 - \frac{1}{10} \nabla f(2x^k_1 - x^k_2)$}
			\State{$x^{k+1}_2 = x^k_1$}
			\EndFor
		\end{algorithmic}
	\end{algorithm}
\end{minipage}
\hfil
\begin{minipage}{0.47\textwidth}
	\begin{algorithm}[H]
		\centering
		\captionsetup{font=scriptsize}
		\caption{}
		\label{algo2}
		\scriptsize
		\begin{algorithmic}
			\For{$k=0, 1, 2,\ldots$}
			\State{$\xi^{k+1}_1 = \xi^k_1 - \xi^k_2 - \frac{1}{5} \nabla f(\xi^k_1)$}
			\State{$\xi^{k+1}_2 = \xi^k_2 + \frac{1}{10} \nabla f(\xi^k_1)$}
			\EndFor
		\end{algorithmic}
	\end{algorithm}
\end{minipage}
\hfil
\vspace{1em}

First consider \cref{algo1,algo2}.
We may transform the iterates of \cref{algo1} by the invertible linear map
$\xi^k_1 = 2x^k_1 - x^k_2, \xi^k_2 = - x^k_1+x^k_2$
to yield the iterates of \cref{algo2}.
Although the iterates are not in exact correspondence as in \crefrange{algo_i1}{algo_i4}, the sequences $(x^k_1)_{k\ge 0}$ and $(x^k_2)_{k\ge 0}$ are
equivalent to the sequences
$(\xi^k_1)_{k\ge 0}$ and $(\xi^k_2)_{k\ge 0}$
up to an invertible linear transformation.

\vspace{-1em}
\noindent
\hfil
\begin{minipage}[t]{0.47\textwidth}
	\begin{algorithm}[H]
		\centering
		\captionsetup{font=scriptsize}
		\caption{}
		\label{algo3}
		\scriptsize
		\begin{algorithmic}
			\For{$k=0, 1, 2,\ldots$}
			\State{${x}^{k+1}_1 = 3{x}^k_1 - 2x^k_2 + \frac{1}{5} \nabla f(-x^k_1 + 2x^k_2)$}
			\State{$x^{k+1}_2 = x^k_1$}
			\EndFor
		\end{algorithmic}
	\end{algorithm}
\end{minipage}
\hfil
\begin{minipage}[t]{0.47\textwidth}
	\begin{algorithm}[H]
		\centering
		\captionsetup{font=scriptsize}
		\caption{}
		\label{algo4}
		\scriptsize
		\begin{algorithmic}
			\For{$k=0, 1, 2,\ldots$}
			\State{$\xi^{k+1} = \xi^k - \frac{1}{5} \nabla f(\xi^k)$}
			\EndFor
		\end{algorithmic}
	\end{algorithm}
\end{minipage}
\hfil
\vspace{1em}

The second example consists of \cref{algo3,algo4}.
\Cref{algo4} is ordinary gradient descent.
These algorithms do not even have the same number of state variables,
so these algorithms are \emph{not} equivalent up to an invertible linear transformation.
But when suitably initialized,
we may transform the iterates of \cref{algo3} by the linear map
$\xi^k = -x^k_1 +2 x^{k}_2$
to yield the iterates of \cref{algo4}.
This transformation is linear but not invertible.
Instead, notice that the sequence of calls to the gradient oracle are identical:
the algorithms satisfy \emph{oracle equivalence},
a notion we will define formally later in this paper. Note that \cref{algo1,algo3} look similar, yet \cref{algo1} is \emph{not} equivalent to gradient descent.

\vspace{-1em}
\noindent
\hfil
\begin{minipage}[t]{0.47\textwidth}
	\begin{algorithm}[H]
		\centering
		\captionsetup{font=scriptsize}
		\caption{(Douglas--Rachford)}
		\label{algo5}
		\scriptsize
		\begin{algorithmic}
			\For{$k=0, 1, 2,\ldots$}
			\State{$x^{k+1}_1 = \prox_{f}(x^k_3)$}
			\State{$x^{k+1}_2 = \prox_{g}(2x^{k+1}_1 - x^k_3)$}
			\State{$x^{k+1}_3 = x^k_3 + x^{k+1}_2 - x^{k+1}_1$}
			\EndFor
		\end{algorithmic}
	\end{algorithm}
\end{minipage}
\hfil
\begin{minipage}[t]{0.47\textwidth}
	\begin{algorithm}[H]
		\centering
		\captionsetup{font=scriptsize}
		\caption{(Simplified ADMM)}
		\label{algo6}
		\scriptsize
		\begin{algorithmic}
			\For{$k=0, 1, 2,\ldots$}
			\State{$\xi^{k+1}_1 = \prox_{g}(\xi_2^k-\xi_3^k)$}
			\State{$\xi^{k+1}_2 = \prox_{f}(\xi^{k+1}_1+\xi^k_3)$}
			\State{$\xi^{k+1}_3 = \xi^k_3+\xi^{k+1}_1-\xi^{k+1}_2$}
			\EndFor
		\end{algorithmic}
	\end{algorithm}
\end{minipage}
\hfil
\vspace{1em}

The third example consists of \cref{algo5,algo6}. These algorithms are known as Douglas--Rachford splitting \cite{lions1979splitting, eckstein1992douglas}
and a special case of the alternating direction method of multipliers (ADMM) \cite[\S8]{ryuyinconvex}\cite{MAL-016}, respectively. 
With suitable initialization,
they will generate the same sequence of
calls to the proximal operators, ignoring the very first call to one of the oracles.
Specifically, \cref{algo6} is initialized as
$\xi^0_2=x_1^1$, $\xi^0_3=x_3^0-x_1^1$
and the first call to $\text{prox}_f$ in \cref{algo5} is ignored.
We will say they are equivalent up to a prefix or shift: they satisfy \emph{shift equivalence}.
We will revisit these algorithms in \cref{shift-equ}.

\vspace{-1em}
\noindent
\hfil
\begin{minipage}{0.47\linewidth}
\begin{algorithm}[H]
	\centering
	\captionsetup{font=scriptsize}
	\caption{(Proximal gradient)}
	\label{algo11x}
	\scriptsize
	\begin{algorithmic}
		\For{$k=0, 1, 2,\ldots$}
		\State{$y^k = x^k-t \nabla f(x^k)$}
		\State{$x^{k+1} = \prox_{tg}(y^k)\vphantom{\prox_{\frac{1}{t}g^*}(\frac{1}{t}y^k)}$}
		\EndFor
	\end{algorithmic}
\end{algorithm}
\end{minipage}
\hfil
\begin{minipage}{0.47\linewidth}
	\begin{algorithm}[H]
	\centering
	\captionsetup{font=scriptsize}
	\caption{(Conjugate proximal gradient)}
	\label{algo12x}
	\scriptsize
	\begin{algorithmic}
		\For{$k=0, 1, 2,\ldots$}
		\State{$y^k = x^k - t\nabla f(x^k)$}
		\State{$x^{k+1} =  y^k -
			t\prox_{\frac{1}{t}g^*}(\frac{1}{t}y^k)$}
		\EndFor
	\end{algorithmic}
\end{algorithm}
\end{minipage}
\hfil
\vspace{1em}

Finally, consider \cref{algo11x,algo12x}. These algorithms do not even call the same oracles; the first algorithm calls $\nabla f$ and $\prox_{tg}$ while the other calls $\nabla f$ and $\prox_{\frac{1}{t}g^*}$ (the proximal operator of the Fenchel conjugate of $g$). Nevertheless, these two oracles are related via Moreau's identity:
$x = \prox_{tg}(x) + t \prox_{\frac{1}{t}g^*}(\tfrac{1}{t}x)$ and applying this identity immediately relates \cref{algo11x,algo12x}. These algorithms satisfy \emph{LFT equivalence} and we will revisit them in \cref{sec:prox_subdiff}.

In \crefrange{oracle-equ}{lft-equ}, we will develop increasingly general notions of equivalence that cover all the motivating examples above and more. 
Before we can formally define algorithm equivalence, we begin by introducing the mathematical representation, borrowed from control theory, that we use to describe iterative algorithms.

%%%%%%%%%%%%%%%%%%%%%%%%%%%%%%%%%%%%%%%%%%%%%%%%%%%%%%%%%%%%%%%%%%%%%%%%%%%%%%%%%%%%%%%%%%%
\section{Preliminaries}\label{preliminary}

We let $\X$ denote a generic real vector space and
$\X^n \defeq \X \times \dots \times \X$ ($n$ times).
We represent $x \in \X^n$ as a column vector with subvectors indexed using subscripts.
In other words, $x = \sbmat{x_1 \\[-1mm]\vdots\\ x_n}$ where $x_1,\dots,x_n\in \X$.
Superscripts are used to index sequences of vectors.
For example, we may write $(x^0,x^1,\dots)$ to denote a semi-infinite sequence of vectors with $x^k \in \X^n$ for each $k\geq 0$.
If $A \in \R^{m\times n}$ and $x \in \X^n$, we overload matrix multiplication by writing $y = Ax \in \X^m$ to mean:
$y_i = \sum_{j=1}^n A_{ij} x_j$ for $i=1,\dots,m$. In this case, we say that $y$ is a \emph{linear function} of $x$.

An \emph{oracle} is a function $\phi: \X \to \X$. We denote the diagonal concatenation of many oracles $(\phi_1,\dots,\phi_p)$ using an upper-case letter: $\Phi: \X^p \to \X^p$, where $u = \Phi(y)$ means that $u_i = \phi_i(y_i)$ for $i=1,\dots,p$.

\titleparagraph{Oracle-based iterative algorithms}

We assume an oracle-based model for our iterative algorithms. The algorithm can query a set of oracles at discrete query points~\cite[\S4]{boyd_vandenberghe_2004}\cite[\S1]{MAL-050}\cite[\S1]{nesterov2018lectures}.
Common examples of oracles include gradients, proximal operators, and projection onto a constraint set~\cite[\S6]{doi:10.1137/1.9781611974997}\cite[\S2]{fenchel1953convex}\cite[\S1]{OPT-003}.
We assume that the oracle outputs are unique and deterministic once any exogenous choices have been fixed. For example, a subgradient oracle might return the subgradient of minimum norm, and a stochastic gradient oracle might return the gradient corresponding to a particular sample path.

For an iterative algorithm that uses oracles $(\phi_1,\dots,\phi_p)$, we assume the following.
\ben
\item The algorithm maintains an internal \emph{state} $x^k \in \X^n$ that is initialized to some $x^0$ before the algorithm begins. 
\item During iteration $k$, each oracle $\phi_i$ is queried exactly once. We call the associated query point $y_i^k \in \X$ and the query result $u_i^k \in \X$. In other words, $u_i^k = \phi_i(y_i^k)$. This convention ensures that query points and results are well defined and unambiguous. If an oracle must be queried multiple times during each iteration, we can simply treat each query as a separate oracle (see \emph{Repeated oracles} in \cref{discussion}). 
\item During iteration $k$, the oracles are queried in a prescribed order $\phi_{i_1},\dots,\phi_{i_p}$.
Each query point $y_{i_j}^k$ is a \emph{linear function} of the state $x^k$ and possibly of the query results $u_{i_1}^k,\dots,u_{i_{j-1}}^k$ obtained thus far. 
\item Once all oracles have been queried, the internal state $x^k$ is updated to $x^{k+1}$ using a \emph{linear function} of $x^k \in \X^n$ and of $u^k \in \X^p$.
\item All aforementioned linear functions are the same at every iteration (independent of $k$). In other words, the algorithm is \emph{time-invariant}.
\een
We will see that this class of algorithms includes commonly used algorithms,
such as accelerated methods, proximal methods, operator splitting methods, and more \cite{hu2020analysis, doi:10.1137/15M1009597}.

Our framework excludes algorithms whose parameters explicitly depend on the iterate index $k$, such as gradient-based methods with diminishing stepsizes.
We view time-varying algorithms as schemes for switching between different time-invariant algorithms.
Thus, in our framework, the notion of algorithm equivalence pertains to the
time-invariant algorithmic components, while the time variation is captured separately
by the switching scheme. Since our aim is to reason about algorithm equivalence, we
therefore restrict attention to time-invariant algorithms.

Here is a pseudo-code implementation of a generic iterative algorithm that satisfies the assumptions above.
\vspace{-1em}
\begin{algorithm}[H]
	\centering
	\caption{Implementation of a generic iterative algorithm}
	\label{algo_generic_iterative}
	\begin{algorithmic}
		\Initialize{$x^0 \in \X^n$}
		\For{$k=0, 1, 2,\ldots$}
			\For{$i=1,\ldots,p$}
				\State{ $y_i^k = \sum_{j=1}^n c_{ij} x_j^k + \sum_{j=1}^{i-1} d_{ij} u_j^k$ }\Comment{Evaluate query point for $i\textsuperscript{th}$ oracle.}
				\State{ $u_i^k = \phi_i(y_i^k)$ }\Comment{Query $i\textsuperscript{th}$ oracle.}
			\EndFor
			\State{$x^{k+1} = A x^k + B u^k$}\Comment{Update internal state.}
		\EndFor
	\end{algorithmic}
\end{algorithm}
\vspace{-1em}

\titleparagraph{State-space form}
We can write the updates in \cref{algo_generic_iterative} in the more compact form
\begin{subequations}\label{eq:ss_generic}
\begin{align}
x^{k+1} &= A x^k + B u^k, \label{ss1}\\
y^k &= C x^k + D u^k, \label{ss2}\\
u^k &= \Phi(y^k),
\end{align}
\end{subequations}
where $A\in\R^{n\times n}$, $B\in\R^{n\times p}$, $C = [c_{ij}] \in\R^{p\times n}$, and $D = [d_{ij}]\in\R^{p\times p}$. The equations \eqref{eq:ss_generic} can also be represented visually using a block diagram, as in \cref{fig_blkdiag_demo2}.

\medskip
\begin{remark}\label{rem:lower-triangular}
	In \cref{algo_generic_iterative}, we assumed the oracles were queried in the order $\phi_1,\dots,\phi_p$, so the $D = [d_{ij}]$ matrix is strictly lower triangular. If the oracles were queried in a different order, the rows and columns of $D$ would be permuted accordingly.
\end{remark}
\medskip

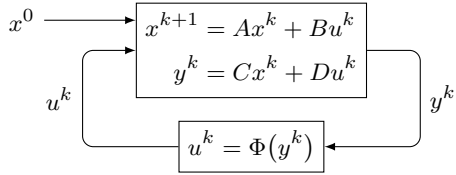
\begin{figure}[ht]
	\centering
	\begin{tikzpicture}[font=\small]
		\tikzstyle{arr}=[>=latex,auto,->,rounded corners]
		\node[draw] at (0, 0) (algo) {$\begin{aligned}
			x^{k+1} &= A x^k + B u^k\\
			y^{k} &= C x^k + D u^k
		\end{aligned}$};
		\node[draw,below= 0.3 of algo] (oracle) {$\begin{aligned}
			u^k = \Phi\bigl(y^k\bigr)\\
			\end{aligned}$};
		\coordinate (t1) at ($(algo.east) + (0.7,0)$);
		\coordinate (t2) at ($(algo.west) + (-0.7,0)$);
		\coordinate (t3) at ($(algo.west) + (-1.5,0.4)$);
		\draw[arr]  (algo) -- (t1) |- node[pos=0.25,align=left]{$y^k$} (oracle);
		\draw[arr]  (oracle) -| node[pos=0.75,align=right]{$u^k$} (t2) -- (algo);
		\node (x0) at (t3) {$x^0$};
		\draw[arr]  (x0) -- (x0 -| algo.west);
	\end{tikzpicture}
	\caption{Block diagram representation of a generic iterative algorithm.}
	\label{fig_blkdiag_demo2}
\end{figure}

The representation of \cref{fig_blkdiag_demo2} separates the \emph{oracles}, which map $y^k \mapsto u^k$, from the \emph{algorithm}, which maps $(x^0,u^0,u^1,\dots,u^k)\mapsto y^k$. This decomposition was first developed in \cite{doi:10.1137/15M1009597}. 
The algorithm is characterized by the matrices $(A,B,C,D)$, which are called a \emph{state-space realization}, and are a widely used representation for linear time-invariant dynamical systems \cite{williams2007linear}. 

We now present a few examples that illustrate how to find a state-space realization.

\titleparagraph{Example: Reflected Gradient Method}

Consider \cref{algo_i4}, which uses oracle $F$ and has update equation
\begin{align}\label{eq:alg_demo_b}
	x_1^{k+1} = x_1^k - \eta F(2x_1^k - x_1^{k-1}).
\end{align}
Since the update for $x_1^{k+1}$ depends on both $x_1^k$ and $x_1^{k-1}$, we augment the internal state to include this past iterate. To this effect, we define $x_2^k \defeq x_1^{k-1}$ and obtain update equations with state $(x_1^k,x_2^k)$ that only depend on the previous timestep:
\begin{align*}
x_1^{k+1} &= x_1^k - \eta F(2x_1^k - x_2^k) \\
x_2^{k+1} &= x_1^k
\end{align*}
Define the oracle query point $y^k$ and query result $u^k$. We can now express \cref{eq:alg_demo_b} in the form of \cref{algo_generic_iterative} and \cref{eq:ss_generic}:
\begin{align*}
&\left.\begin{aligned}
y^k &= 2x_1^k - x_2^k \\
u^k &= F(y^k) 
\end{aligned}\;\;\right\} & y^k &= \bmat{2 & -1}\bmat{x_1^k \\ x_2^k} + \bmat{0}u^k\\
&\left.\begin{aligned}
x_1^{k+1} &= x_1^k - \eta F(2x_1^k-x_2^k) = x_1^k - \eta u^k \\
x_2^{k+1} &= x_1^k
\end{aligned}\;\;\right\} & \bmat{x_1^{k+1} \\ x_2^{k+1}} &= \bmat{1 & 0 \\ 1 & 0}\bmat{x_1^k \\ x_2^k} + \bmat{-\eta \\ 0}u^k
\end{align*}
Therefore, a state-space realization for \cref{algo_i4} is given by
\begin{equation}\label{eq:ss_admm_b}
	(A,B,C,D) = \left( \bmat{1 & 0 \\ 1 & 0}, \bmat{-\eta \\ 0}, \bmat{2 & -1}, \bmat{0} \right).	
\end{equation}

\titleparagraph{Example: simplified ADMM}

Consider \cref{algo6} (simplified ADMM), 
which uses state variables $(\xi_1^k,\xi_2^k,\xi_3^k)$, 
oracles $(\prox_{f},\prox_{g})$, and update equations
\begin{align}\label{eq:alg_demo}
	\xi^{k+1}_1 &= \prox_{g}(\xi_2^k-\xi_3^k), &
	\xi^{k+1}_2 &= \prox_{f}(\xi^{k+1}_1+\xi^k_3), &
	\xi^{k+1}_3 &= \xi^k_3+\xi^{k+1}_1-\xi^{k+1}_2.
\end{align}
Define the oracle query points $(y_1^k,y_2^k)$ and query results $(u_1^k,u_2^k)$. We can now express \cref{eq:alg_demo} in the form of \cref{algo_generic_iterative} and \cref{eq:ss_generic}:
\begin{align*}
&\left.\begin{aligned}
y^k_2 &= \xi_2^k - \xi_3^k \\
u_2^k &= \prox_g(y_2^k) \\
y^k_1 &= \xi_1^{k+1} + \xi_3^k = \xi_3^k + u_2^k \\
u_1^k &= \prox_f(y_1^k)
\end{aligned}\;\;\right\} &
\bmat{y^k_1 \\ y^k_2}
&\!=\! \bmat{0 & 0 & 1 \\ 0 & 1 & -1}\!\!\bmat{\xi^k_1 \\ \xi^k_2 \\ \xi^k_3}
\!+\! \bmat{0 & 0 \\ 0 & 1}\!\!\bmat{u^k_1 \\ u^k_2} \\
&\left.\begin{aligned}
\xi_1^{k+1} &= u_2^k \\
\xi_2^{k+1} &= u_1^k \\
\xi_3^{k+1} &= \xi_3^k + \xi_1^{k+1} - \xi_2^{k+1} = \xi_3^k + u_2^k - u_1^k
\end{aligned}\;\;\right\} &
\bmat{\xi^{k+1}_1 \\ \xi^{k+1}_2 \\ \xi^{k+1}_3}
&\!=\! \bmat{0 & 0 & 0 \\ 0 & 0 & 0 \\ 0 & 0 & 1}\!\!\bmat{\xi^k_1 \\ \xi^k_2 \\ \xi^k_3}
\!+\! \bmat{0 & 1 \\ 1 & 0 \\ -1 & 1}\!\!\bmat{u^k_1 \\ u^k_2}
\end{align*}
Therefore, a state-space realization for \cref{algo6} is given by
\begin{equation}\label{eq:ss_admm}
	(A,B,C,D) = \left( \bmat{0 & 0 & 0\\0 & 0 & 0\\0 & 0 & 1}, \bmat{0 & 1 \\ 1 & 0\\-1 & 1}, \bmat{0 & 0 & 1\\ 0 & 1 & -1}, \bmat{0 & 1 \\ 0 & 0} \right).	
\end{equation}

\titleparagraph{Explicit and implicit implementations}

Given a state-space realization $(A,B,C,D)$ where $A\in\R^{n\times n}$, $B\in\R^{n\times p}$, $C\in\R^{p\times n}$, and $D\in\R^{p\times p}$, when is it possible to construct a corresponding step-by-step implementation in the form of \cref{algo_generic_iterative}?
By \cref{rem:lower-triangular}, it is possible provided there exists a permutation matrix $P$ such that $P^\top D P$ is strictly lower-triangular. The permutation $P$ describes the order in which the oracles $\phi_1,\dots,\phi_p$ will be evaluated in \cref{algo_generic_iterative}. Put another way, if we view $D$ as the adjacency matrix for a directed graph, the corresponding graph should be a \emph{directed acyclic graph} (DAG).

If the $D$ matrix does \emph{not} correspond to a DAG, the graph will exhibit a cycle, which will manifest itself as an implicit equation involving oracles.\footnote{also known as an ``algebraic loop'' or a ``circular dependency''.} For example, consider the realization
$(A,B,C,D) = (1,-t,1,-t)$. This algorithm has the update equation
\begin{align}\label{eq:implicit}
	\left.\begin{aligned}
		x^{k+1} &= x^k - t u^k \\
		y^k &= x^k - t u^k \\
		u^k &= \phi(y^k)
	\end{aligned}\quad\right\}
	\quad\iff\quad
	x^{k+1} = x^k - t \phi(x^{k+1}).
\end{align}
The $D$ matrix is not strictly lower-triangular and $x^{k+1}$ is defined \emph{implicitly}.

\medskip
\begin{definition}\label{def:explicit_implicit}
	If the state-space realization $(A,B,C,D)$ for an algorithm has the property that there exists a permutation matrix $P$ such that $P^\top DP$ is strictly lower-triangular, we say that the algorithm has an \emph{explicit} implementation. Otherwise, we say it has an \emph{implicit} implementation. 
\medskip
\end{definition}

The same algorithm may have an explicit implementation with one oracle and an implicit implementation with another oracle. For example, consider \cref{eq:implicit} and let $\phi = \nabla f$, where $f$ is convex. Then, by the first-order optimality conditions, we have
\begin{align*}
x^{k+1} &= \prox_{tf}(x^k) &\iff&& x^{k+1} &= \argmin_x \Bigl( f(x) + \tfrac{1}{2t}\norm{x-x^k}^2 \Bigr) \\
& &\iff&& x^{k+1} &= x^k - t \nabla f(x^{k+1}).
\end{align*}
In other words, using the oracle $\prox_{tf}$ yields an explicit implementation, but using the oracle $\nabla f$ yields an implicit implementation. We will see later (\cref{sec:prox_subdiff}) how to automatically detect such equivalences.

State-space realizations conveniently parametrize a large class of explicit and implicit
algorithms in terms of matrices $(A,B,C,D)$, but the representation is not unique.
For example, \crefrange{algo_i1}{algo_i4} have different realizations $(A,B,C,D)$
despite having identical state sequences.
In the next section, we show how tools from control theory can clarify the 
relations between these sorts of representations.

%%%%%%%%%%%%%%%%%%%%%%%%%%%%%%%%%%%%%%%%%%%%%%%%%%%%%%%%%%%%%%%%%%%%%%%%%%%%%%%%%%%%%%%%%%%
\section{Algorithm representation}\label{control}

In this section, we explain how to represent algorithms using \emph{transfer functions}, a standard tool in linear systems and control theory \cite[\S1--3]{antsaklis2006linear}\cite[\S1,2,5]{williams2007linear}. We will give an overview of relevant terminology and show how to convert an algorithm to and from the transfer function representation.

In \cref{preliminary}, we discussed algorithms that have a state-space realization
\begin{subequations}\label{eq:ss}
\begin{align}
	x^{k+1} &= Ax^k + Bu^k, \\
	y^k &= Cx^k + Du^k.
\end{align}
\end{subequations}
We represent semi-infinite sequences such as $(x^0,x^1,\dots)$ using their \emph{$z$-transforms}. That is, we define the formal power series\footnote{The use of $z^{-1}$ as the variable in the $z$-transform is a common convention in control theory.}
\[
\mathcal{Z}[x^k] = \hat x(z) \defeq \sum_{k=0}^\infty x^k z^{-k}
\]
and similarly for $\hat u(z)$ and $\hat y(z)$. When taking the $z$-transform of the forward-shifted sequence $(x^1,x^2,\cdots)$, we have:
\[
\mathcal{Z}\bigl[x^{k+1}\bigr] = \sum_{k=0}^\infty x^{k+1} z^{-k} = z (\hat x(z) - x^0) 
\]
Evaluating the $z$-transform of \eqref{eq:ss}, we obtain:
\begin{subequations}\label{eq:ssz}
	\begin{align}
		z(\hat x(z) - x^0) &= A \hat x(z) + B\hat u(z), \\
		\hat y(z) &= C \hat x(z) + D \hat u(z).
	\end{align}
\end{subequations}
Eliminating $\hat x(z)$ from \eqref{eq:ssz}, we obtain
\begin{equation}\label{eq:zdomain_yu_map}
\hat y(z) = \underbrace{ z C(zI-A)^{-1} }_{\hat O(z)} x^0
+ \underbrace{ \left( D + C(zI-A)^{-1}B \right)}_{\hat H(z)} \hat u(z)
\end{equation}
The (matrix-valued) functions $\hat O(z)$ and $\hat H(z)$ are convenient to work with because they relate $\hat y(z)$ and $\hat u(z)$ via conventional matrix multiplication.
The function $\hat H(z)$ is called the \emph{transfer function}, and this is how we will represent the state-space system $(A,B,C,D)$. 
We use the special notation
\begin{equation}\label{eqH}
	\hat H(z) = \left[\begin{array}{c|c}A&B\\ \hline C&D\end{array}\right]
	\defeq D + C(zI-A)^{-1} B.
\end{equation}
Since these transfer functions arise from linear state-space systems, they are rational matrix-valued functions of z, so the identities required for our equivalence tests reduce to exact algebraic identities between rational functions.

For ease of notation, we will often omit the ``$(z)$'' after each transfer function, so when we write $\hat H_1 = \hat H_2$, we mean that $\hat H_1(z) = \hat H_2(z)$ for all $z$.
We will also overload oracles so that they may apply to the $z$-transforms directly by threading across coefficients. Namely, if $\Phi(y^k) = u^k$ for $k=0,1,\dots$, we will write:
\[
\Phi(\hat y) \defeq \sum_{k=0}^\infty \Phi(y^k) z^{-1} = \sum_{k=0}^\infty u^k z^{-1} = \hat u
\]
Therefore, the block diagram of \cref{fig_blkdiag_demo2} can be written in terms of transfer functions and $z$-transforms as in \cref{fig_blkdiag_demo3}.

\begin{figure}[ht]
	\centering
	\begin{tikzpicture}[font=\small]
		\tikzstyle{arr}=[>=latex,auto,->,rounded corners]
		\node[draw,minimum height=1cm] at (0, 0) (algo) {$\begin{aligned}
			\hat y = \hat O x^0 + \hat H \hat u
		\end{aligned}$};
		\node[draw,below= 0.3 of algo] (oracle) {$\begin{aligned}
			\hat u = \Phi\bigl(\hat y\bigr)\\
			\end{aligned}$};
		\coordinate (t1) at ($(algo.east) + (0.7,0)$);
		\coordinate (t2) at ($(algo.west) + (-0.7,-0.2)$);
		\coordinate (t3) at ($(algo.west) + (-1.5,0.2)$);
		\draw[arr]  (algo) -- (t1) |- node[pos=0.25,align=left]{$\hat y$} (oracle);
		\draw[arr]  (oracle) -| node[pos=0.75,align=right]{$\hat u$} (t2) -- (t2 -| algo.west);
		\node (x0) at (t3) {$x^0$};
		\draw[arr]  (x0) -- (x0 -| algo.west);
	\end{tikzpicture}
	\caption{Block diagram representation of a generic optimization algorithm expressed in terms of its transfer function $\hat H$ and the $z$-transforms of its inputs and outputs.}
	\label{fig_blkdiag_demo3}
\end{figure}
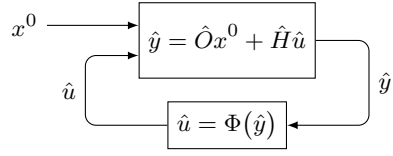

Applying the formula \eqref{eqH} to the state-space matrices of simplified ADMM (\cref{algo6}) from \eqref{eq:ss_admm}, we obtain the transfer function
\begin{equation}\label{eq:admmtf}
\hat H_{\ref{algo6}} = 
\bmat{\frac{-1}{z-1} & \frac{z}{z-1} \\
\frac{2 z-1}{z(z-1)} & \frac{-1}{z-1}}.
\end{equation}
The transfer function can be readily computed directly from the update equations by replacing each state by its $z$-transform, neglecting initial conditions, and eliminating the state variables. The following Python code computes the transfer function for \cref{algo6} (\cref{eq:admmtf}) starting from the update equations.

\smallskip
\begin{lstlisting}[style=mypython]
from sympy import Eq, var, solve, simplify, linear_eq_to_matrix

# Define the symbols
var('z xi1 xi2 xi3 y1 y2 u1 u2')

# Unknowns to solve for
unknowns = (xi1, xi2, xi3, y1, y2)

# State-update and output equations
eqns = [
  Eq(z*xi1, u2),                  # xi1[k+1] = u2[k]
  Eq(z*xi2, u1),                  # xi2[k+1] = u1[k]
  Eq(z*xi3, xi3 + z*xi1 - z*xi2), # xi3[k+1] = xi3[k] + xi1[k+1] - xi2[k+1]
  Eq(y1, z*xi1 + xi3),            #    y1[k] = xi1[k+1] + xi3[k]
  Eq(y2, xi2 - xi3),              #    y2[k] = xi2[k] - xi3[k]
]

# Solve for states and outputs in terms of inputs
sol = solve(eqns, unknowns)

# Algorithm inputs and outputs
inputs = (u1, u2)
outputs = (sol[y1], sol[y2])

# Extract transfer matrix H from y = H u
H, _ = linear_eq_to_matrix(outputs, inputs)
H = simplify(H)
\end{lstlisting}

\smallskip

\noindent\cref{eq:zdomain_yu_map} shows that state-space systems have two components:
\begin{enumerate}
	\item The map from initial state to output, $\hat O$.
	\item The transfer function, $\hat H$.
\end{enumerate}
We argue that the transfer function $\hat H$ alone is a sufficiently rich representation for the purpose of evaluating the equivalence of state-space systems. Roughly, when two state space systems have the same transfer function, we can find initial conditions that cause the systems to have \emph{identical input-output maps}. In other words, from the perspective of the oracle, the algorithms are indistinguishable. We state this result as \cref{prop:tf-ss} below.

\medskip
\begin{proposition}\label{prop:tf-ss}
	Suppose system $i\in\{1,2\}$ has state-space realization $(A_i,B_i,C_i,D_i)$, initial state $x_i^0$, and associated transfer function $\hat H_i$. The following are equivalent.
	\begin{enumerate}
		\item $\hat H_1 = \hat H_2$.
		\item There exist $x_1^0$ and $x_2^0$ such that
		both systems have the same input-output map. \label{it:prop1}
	\end{enumerate}
\end{proposition}

Under additional mild assumptions about the state-space realization, we can strengthen the forward implication of \cref{prop:tf-ss} to include any initial condition.

\medskip
\begin{proposition}\label{prop:tf-ss2}
Consider the setting of \cref{prop:tf-ss} and further assume that both systems have \emph{minimal} state-space realizations. The following are equivalent.
\begin{enumerate}
\item $\hat H_1 = \hat H_2$.
\item For every initialization of one system, there exists a unique initialization of the other system such that both systems have the same input-output map. 
\end{enumerate}
\end{proposition}

\noindent
\emph{Minimality} (see \cref{app:minimality}) means that the realization contains no redundant internal state. Equivalently, among all state-space realizations that induce the same transfer function, a minimal
realization has the smallest possible state dimension. In the present context, this means that the realization stores exactly the internal memory needed to reproduce the same oracle-input/output behavior, and no more.
Minimality is a mild assumption because from any non-minimal realization, one can construct a minimal realization with the same transfer function.

For proofs of \cref{prop:tf-ss,prop:tf-ss2} and further details on how to construct minimal realizations, see \cref{proof:prop12,app:construction}, respectively. 

\titleparagraph{From transfer functions to algorithms}

In this paper, we study the equivalence of algorithms by analyzing their transfer functions. We always start with update equations, which lead to state-space realizations, which lead to transfer functions. For completeness, we also provide a complete characterization of when the process can be reversed,
including a method to construct the update equations from the transfer function when possible, in \cref{app:construction}.

%%%%%%%%%%%%%%%%%%%%%%%%%%%%%%%%%%%%%%%%%%%%%%%%%%%%%%%%%%%%%%%%%%%%%%%%%%%%%%%%%%%%%%%%%%%
\subsection{Algorithm equivalence}\label{equivalence}

In the following three sections, we propose three notions of algorithm equivalence, each increasingly more general.
\begin{enumerate}
	\item \emph{Oracle equivalence.} For use when comparing algorithms that each use the same single oracle.
	\item \emph{Shift equivalence.} For use when comparing algorithms that each use the same set of oracles. Oracle equivalence is a special case of shift equivalence.
	\item \emph{LFT equivalence.} For use when comparing algorithms that use oracles that are related via a linear fractional transforms (LFTs), which we define in \cref{lft-equ}. Oracle equivalence and shift equivalence are both special cases of LFT equivalence.
\end{enumerate}
Although the above notions of equivalence cover all examples of algorithm equivalence we have observed in practice, there are limitations to our definitions. We discuss limitations and possible extensions in \cref{discussion}.

%%%%%%%%%%%%%%%%%%%%%%%%%%%%%%%%%%%%%%%%%%%%%%%%%%%%%%%%%%%%%%%%%%%%%%%%%%%%%%%%%%%%%%%%%%%
\section{Oracle equivalence}\label{oracle-equ}

The idea behind oracle equivalence is to ask the question: ``are the algorithms indistinguishable from the point of view of the oracle?''
In other words, if the algorithms are suitably initialized, would using the same algorithm inputs (oracle outputs) $(u^0,u^1,\dots)$ for both algorithms always produce the same algorithm outputs (oracle inputs) $(y^0,y^1,\dots)$?
If the answer is ``yes'', then the algorithms are \emph{oracle equivalent}.

Motivated by \cref{prop:tf-ss,prop:tf-ss2}, we will formally define oracle equivalence using the notion of transfer functions.

\medskip
\begin{definition}[oracle equivalence]\label{def1}
	Two algorithms that use the same oracles are \emph{oracle equivalent} if they have have the same transfer function.	
\end{definition}
\medskip

\noindent Oracle equivalence is a useful notion of algorithm equivalence:

\begin{enumerate}
\item There are many ways to re-parameterize an algorithm that change the state-space matrices $(A,B,C,D)$ or the internal state $x^k$. A sensible notion of equivalence should be independent of such transformations. Oracle equivalence achieves this independence by bypassing the state entirely and treating the algorithm as the map $(u^0,u^1,\dots)\mapsto(y^0,y^1,\dots)$.

\item Since oracle-equivalent algorithms generate identical oracle-input and oracle-output sequences, many analytical properties of interest are preserved, especially those commonly studied for optimization algorithms. For example, if the oracle is the gradient of a differentiable function, $u^k = \nabla f(y^k)$, then any quantity determined solely by the gradient queries and responses evolves identically for the two algorithms, including the sequence of gradient norms $\|\nabla f(y^k)\|$, the sequence of queried objective values $f(y^k)$, and any certificate or bound expressed in terms of these quantities.
\end{enumerate}

\titleparagraph{Invariance under linear state transformations}

Oracle equivalence (\cref{def1}) is invariant under linear transformations of state. Specifically, define $\tilde{x}^k = Tx^k$ for each $k$, where $T$ is an invertible matrix.
The state-space equations \eqref{eq:ss} expressed in terms of the new state variable $\tilde x^k$ become
\begin{subequations}\label{eqp11}
	\begin{align}
		\tilde x^{k+1}  & = TAT^{-1} \tilde x^k + TB u^k, \\
		y^k & = CT^{-1} \tilde x^k + Du^k.
	\end{align}
\end{subequations}
Substituting into \cref{eqH}, we can verify that both systems have the same transfer function.\footnote{Both systems will also have the same \emph{Markov parameters} and \emph{Hankel matrices} (see \cref{app:minimality}).} If we initialize the transformed system with $\tilde x^0 = T x^0$ and apply the same input $(u_0,u_1,\dots)$ to both systems, we will obtain the same output $(y_0,y_1,\dots)$, although the respective states $x^k$ and $\tilde x^k$ will generally be different.
This invariance is the key to understanding when two optimization algorithms are the same, even if they look different as written.
For example, this idea alone suffices to show that \cref{algo1,algo2} are equivalent. 

When using linear state transformations, the number of states (size of the $A$ matrix) is preserved. However, realizations with a different number of states can also be oracle equivalent. Although a realization with a larger $A$ matrix will generally lead to a transfer function with higher degree (via \cref{eqH}), there may be common factors that cancel from the numerator and denominator, leading to lower-degree transfer functions that could have been obtained from a realization with a smaller $A$ matrix. This idea is related to the notion of minimality (see \cref{app:minimality}) and explains why \cref{algo3,algo4} are equivalent.

\subsection{Examples of oracle equivalence}\label{charac-oracle}
Now, we will revisit the first and second motivating examples
and apply \cref{def1} to show oracle equivalence. Specifically, we will compute transfer functions using \cref{eqH} and verify that they are the same.

\titleparagraph{\texorpdfstring{\cref{algo1,algo2}}{Algorithms \ref{algo1} and \ref{algo2}}}
The state-space realization and transfer function of \cref{algo1,algo2} are
\begin{align*}
\hat H_{\ref{algo1}} &= \left[\begin{array}{c c|c}
	2 & -1 &-\frac{1}{10}\\
	1 & 0 & 0  \\
	\hline
	2 & -1 & 0
\end{array}\right] = \bmat{2 & -1 } \left(zI - \bmat{2 & -1 \\1 & 0}\right)^{-1} \bmat{-\frac{1}{10}\\  0} = \frac{-2z + 1}{10(z-1)^2},\\
\hat H_{\ref{algo2}} &= \left[\begin{array}{c c|c}
		1 & -1 &-\frac{1}{5}\\
		0 & 1 & \frac{1}{10}  \\
		\hline
		1 & 0 & 0
	\end{array}\right] =
	 \bmat{1 & 0} \left(zI - \bmat{1 & -1 \\ 0 & 1 }\right)^{-1} \bmat{-\frac{1}{5}\\  \frac{1}{10}} = \frac{-2z + 1}{10(z-1)^2}.
\end{align*}
Since $\hat H_{\ref{algo1}} = \hat H_{\ref{algo2}}$, \cref{algo1,algo2} are
oracle-equivalent by \cref{def1}. \cref{algo1} can also be transformed to \cref{algo2} as in \cref{eqp11} via $T = \sbmat{ 2 & -1 \\ -1 & 1}$.

\titleparagraph{\texorpdfstring{\cref{algo3,algo4}}{Algorithms \ref{algo3} and \ref{algo4}}}
The state-space realization and transfer function of \cref{algo3,algo4} are
\begin{align*}
\hat H_{\ref{algo3}} &= \left[\begin{array}{c c|c}
	3 & -2 & \frac{1}{5}\\
	1 & 0 & 0  \\
	\hline
	-1 & 2 & 0
\end{array}\right] =
\bmat{-1 & 2 } {\left(zI - \bmat{3 & -2 \\1 & 0}\right)}^{-1} \bmat{\frac{1}{5}\\  0}
		% = \frac{-(z-2)}{5(z-1)(z-2)}
		= \frac{-z+2}{5(z^2-3z+2)}
		=\frac{-1}{5(z-1)},\\
\hat H_{\ref{algo4}} &= \left[\begin{array}{c|c}
		1 & -\frac{1}{5}\\
		\hline
		1 &  0
	\end{array}\right] = \bmat{1} \left(zI - \bmat{1}\right)^{-1}
	\bmat{-\frac{1}{5}} = \frac{-1}{5(z-1)}.
\end{align*}
Since $\hat H_{\ref{algo3}} = \hat H_{\ref{algo4}}$, \cref{algo3,algo4} are
oracle-equivalent by \cref{def1}. The transfer functions are the same due to the cancellation of the common factor $(z-2)$ in the numerator and denominator of $\hat H_{\ref{algo3}}$.

\titleparagraph{\texorpdfstring{\Crefrange{algo_i1}{algo_i4}}{\ref{algo_i1}--\ref{algo_i1}}}

Using the same approach as above, we can derive the transfer functions for \crefrange{algo_i1}{algo_i4} and show that they are all equal to $\hat H(z) = -\frac{\eta (2z-1)}{z(z-1)}$. Therefore, \crefrange{algo_i1}{algo_i4} are oracle-equivalent by \cref{def1}.

\bigskip

\titleparagraph{Accelerated gradient methods}

Accelerated gradient methods are a class of optimization algorithms designed to improve the convergence speed of gradient-based methods, especially for convex optimization problems. Accelerated methods incorporate momentum-like terms and interpolated iterates that help the optimization process converge faster. Two well-know methods include Polyak's Heavy Ball (HB) \cite{polyak1964some} and Nesterov's Accelerated Gradient Method (NAG) \cite{nesterov1983method}, shown below as \cref{algo_heavyball,algo_nesterov}. These techniques and their stochastic variants are widely used in machine learning, signal processing, and numerical optimization.

\vspace{-1em}
\noindent\hfil\begin{minipage}[t]{0.47\textwidth}
	\begin{algorithm}[H]
		\centering
		\captionsetup{font=scriptsize}
		\caption{(Polyak's Heavy Ball)}
		\label{algo_heavyball}
		\scriptsize
		\begin{algorithmic}
			\For{$k=0,1, 2,\dots$}
			\State{$x^{k+1} = x^k - \alpha \nabla f(x^k) + \beta(x^k-x^{k-1})$}
			\EndFor
		\end{algorithmic}
	\end{algorithm}
\end{minipage}
\hfil
\begin{minipage}[t]{0.47\textwidth}
	\begin{algorithm}[H]
		\centering
		\captionsetup{font=scriptsize}
		\caption{(Nesterov's Method)}
		\label{algo_nesterov}
		\scriptsize
		\begin{algorithmic}
			\For{$k=0,1, 2,\dots$}
			\State{$y^k = x^k +\beta(x^k-x^{k-1})$}
			\State{$x^{k+1} = y^k - \alpha \nabla f(y^k)$}
			\EndFor
		\end{algorithmic}
	\end{algorithm}
\end{minipage}
\hfil

\vspace{1em}
Several works have proposed \emph{unified} momentum algorithms and associated analyses that generalize HB and NAG and allow the algorithm designer to interpolate between both algorithms. Examples include: 
Triple Momentum Method \cite{van2017fastest},
Quasi-Hyperbolic Momentum \cite{ma2019qh},
Stochastic Unified Method \cite{yan2018unified}, and
Unified Stochastic Momentum \cite{shen2023unified},
listed below as \crefrange{algo_TMM1}{algo_TMM4}, respectively.

\vspace{-1em}
\noindent\hfil\begin{minipage}[t]{0.47\textwidth}
	\begin{algorithm}[H]
		\centering
		\captionsetup{font=scriptsize}
		\caption{(Triple Momentum Method)}
		\label{algo_TMM1}
		\begin{algorithmic}
			\For{$k=0,1,2,\dots$}
			\State{$y^k = x^k + \gamma(x^k - x^{k-1})$}
			\State{$x^{k+1} = x^k - \alpha \nabla f(y^k) + \beta(x^k-x^{k-1})$}
			\EndFor
		\end{algorithmic}
	\end{algorithm}
\end{minipage}
\hfil
\begin{minipage}[t]{0.47\textwidth}
	\begin{algorithm}[H]
		\centering
		\captionsetup{font=scriptsize}
		\caption{(Quasi-Hyperbolic Momentum)}
		\label{algo_TMM2}
		\begin{algorithmic}
			\For{$k=0,1,2,\dots$}
			\State{$g^{k+1} = \beta g^k + (1-\beta)\nabla f(\theta^k)$}
			\State{$\theta^{k+1} = \theta^k -\alpha((1-\nu)\nabla f(\theta^k)+\nu g^{k+1})$}
			\EndFor
		\end{algorithmic}
	\end{algorithm}
\end{minipage}
\hfil

\vspace{-1em}
\noindent\hfil\begin{minipage}[t]{0.47\textwidth}
	\begin{algorithm}[H]
		\centering
		\captionsetup{font=scriptsize}
		\caption{(Stochastic Unified Momentum)}
		\label{algo_TMM3}
		\begin{algorithmic}
			\For{$k=0,1,2,\dots$}
			\State{$y^{k+1} = x^k - \alpha \nabla f(x^k)$}
			\State{$q^{k+1} = x^k - s\alpha\nabla f(x^k)$}
			\State{$x^{k+1}=y^{k+1}+\beta(q^{k+1}-q^k)$}
			\EndFor
		\end{algorithmic}
	\end{algorithm}
\end{minipage}
\hfil
\begin{minipage}[t]{0.47\textwidth}
	\begin{algorithm}[H]
		\centering
		\captionsetup{font=scriptsize}
		\caption{(Unified Stochastic Momentum)}
		\label{algo_TMM4}
		\begin{algorithmic}
			\For{$k=0,1,2,\dots$}
			\State{$m^k = \mu m^{k-1}-\eta \nabla f(x^k)$}
			\State{$x^{k+1}=x^{k}+m^k+\lambda\mu(m^k-m^{k-1})$}
			\EndFor
		\end{algorithmic}
	\end{algorithm}
\end{minipage}
\hfil

\vspace{1em}
\noindent The transfer functions for HB and NAG are clearly different:
\begin{align*}
\hat H_{\ref{algo_heavyball}} &= \frac{-\alpha z}{(z-1)(z-\beta)},&
\hat H_{\ref{algo_nesterov}} &= \frac{-\alpha(1+\beta)(z-\frac{\beta}{1+\beta})}{(z-1)(z-\beta)}.
\end{align*}
However, the transfer functions for \crefrange{algo_TMM1}{algo_TMM4} are:
\begin{align*}
\hat H_{\ref{algo_TMM1}}(z) &= -\frac{\alpha(1+\gamma)  (z-\frac{\gamma}{1+\gamma})}{(z-1) (z-\beta )}, &
\hat H_{\ref{algo_TMM2}}(z) &= -\frac{\alpha(1-\beta\nu)  (z-\frac{\beta(1-\nu)}{1-\beta\nu})}{(z-1) (z-\beta )}, \\
\hat H_{\ref{algo_TMM3}}(z) &= -\frac{\alpha(1+\beta s)  (z-\frac{\beta s}{1+\beta s})}{(z-1) (z-\beta )}, &
\hat H_{\ref{algo_TMM4}}(z) &= -\frac{\eta(1+\lambda \mu) (z-\frac{\lambda \mu}{1+\lambda \mu})}{(z-1) (z-\mu )}.
\end{align*}
Each of the above transfer functions are of the form $-\frac{a(z-c)}{(z-1)(z-b)}$ for some $a,b,c$ and can be made equal to one another (oracle equivalent) via suitable choices of the algorithm parameters. In other words, these algorithms parameterize the same space of possible algorithms (which also includes HB and NAG); they are equally general.

\titleparagraph{Distributed optimization}

For our final example of this section, we consider \emph{synchronous distributed optimization}, where a network of computing nodes work collaboratively to solve the problem
	\begin{align}\label{opt:distrop}
	\underset{x\in\R^d}{\text{minimize}} \quad & \sum_{i=1}^n f_i(x).
	\end{align}
Each node $i$ maintains a local state $x_i^k \in \R^d$ and can access the oracle $\nabla f_i$. At every timestep, each node can \emph{gossip} (obtain the local states of its neighboring nodes $x_j^k$), evaluate its local oracle, and perform computations to update its local state. There are two goals: (1) \emph{consensus}: the nodes' local states should converge to a common value, and (2) \emph{optimality}: the common value should be $x^\star$, a solution of \cref{opt:distrop}. For convenience, we use the shorthand notation
\[
x^k \defeq \bmat{x_1^k \\\vdots \\ x_n^k}
\quad\text{and}\quad
\nabla f(x^k) \defeq \bmat{ \nabla f_1(x_1^k) \\ \vdots \\ \nabla f_n(x_n^k)}.	
\]
Gossip is modeled as matrix multiplication $W x^k$, where $W = \tilde W \otimes I_d$ and $\tilde W \in \R^{n\times n}$ is a (typically sparse) row-stochastic matrix; it satisfies $0 \leq \tilde W \leq 1$ and $W\mathbf{1} = \mathbf{1}$.

Under suitable assumptions on $W$, the algorithm $x^{k+1} = W x^k$ achieves consensus at a linear rate, but the consensus value will be the mean of the $x_i^0$ rather than $x^\star$ (no optimality). Likewise, if the $f_i$ are smooth and strongly convex, gradient descent $x^{k+1} = x^k - \alpha \nabla f(x^k)$ achieves local but not global optimality: 
each node converges to the minimizer of its local $f_i$ rather than 
the minimizer of $\sum_{i=1}^n f_i$.
A simple algorithm that achieves both consensus and optimality is distributed gradient descent \cite{DGD}, which combines features of both gossip and gradient descent: $x^{k+1} = W x^k - \alpha_k \nabla f(x^k)$. 
However, this algorithm only converges sublinearly, 
even with strongly convex $f_i$, and requires a diminishing stepsize $\alpha_k \to 0$ to converge at all.

The first algorithm to guarantee linear convergence for strongly convex $f_i$ was 
EXTRA \cite{EXTRA}, and since then many papers have developed new algorithms or refined existing ones to solve \cref{opt:distrop} with a linear convergence rate. 
Two such well-known algorithms are NIDS \cite{NIDS} and Exact Diffusion \cite{ExDIFF}, shown below.

\vspace{-1em}
\noindent\hfil
\begin{minipage}[t]{0.6\linewidth}
\begin{algorithm}[H]
	\captionsetup{font=scriptsize}
	\caption{NIDS}
	\label{algoNIDS}
	\scriptsize
	\begin{algorithmic}
		\For{$k=0, 1, 2,\ldots$}
		\State{$x^{k+2} = W\bigl(2x^{k+1} -x^k-\alpha \nabla f(x^{k+1})+\alpha \nabla f(x^k)\bigr)$}
		\EndFor
	\end{algorithmic}
\end{algorithm}
\end{minipage}
\hfil
\begin{minipage}[t]{0.38\linewidth}
\begin{algorithm}[H]
	\captionsetup{font=scriptsize}
	\caption{Exact Diffusion}
	\label{algoExDIFF}
	\scriptsize
	\begin{algorithmic}
		\For{$k=0, 1, 2,\ldots$}
		\State{$x^{k+1} = w^k -\alpha \nabla f(w^k)$}
		\State{$y^k = x^{k+1} + w^k - x^k$}
		\State{$w^{k+1} = W y^k$} 
		\EndFor
	\end{algorithmic}
\end{algorithm}
\end{minipage}
\vspace{1em}

These algorithms were developed using different approaches. NIDS used a \emph{gradient-differencing} intuition similar to EXTRA to achieve linear convergence (storing the past gradient and updating based on the difference, as in \cref{algoNIDS}). In contrast, Exact Diffusion used an \emph{adapt-correct-combine} concept (corresponding to the three update equations in \cref{algoExDIFF}, respectively).

However, NIDS and Exact Diffusion are (oracle) equivalent \cite{canform_acc}! 
We can detect this equivalence automatically by computing the transfer function for each algorithm. In this case, we obtain
\(
\hat H_{\ref{algoNIDS}}(z) = \hat H_{\ref{algoExDIFF}}(z) = -\alpha (z-1) W (z^2 I - 2zW + W)^{-1}
\).

%%%%%%%%%%%%%%%%%%%%%%%%%%%%%%%%%%%%%%%%%%%%%%%%%%%%%%%%%%%%%%%%%%%%%%%%%%%%%%%%%%%%%%%%%%%
\section{Shift equivalence}\label{shift-equ}

Now consider \cref{algo5,algo6} from the third motivating example. We can calculate that the algorithms have different transfer functions:
\begin{equation}\label{ex:shift}
\hat H_{\ref{algo5}} = \bmat{
		\frac{-1}{z-1} & \frac{1}{z-1} \\
		\frac{2 z-1}{z-1} & \frac{-1}{z-1}},\qquad
\hat H_{\ref{algo6}} = \bmat{
		\frac{-1}{z-1} & \frac{z}{z-1} \\
		\frac{2 z-1}{z(z-1)} & \frac{-1}{z-1}},
\end{equation}
so they are not oracle-equivalent. We can represent the equations for \cref{algo5,algo6} using block diagrams that are \emph{unrolled in time}; see \cref{fig:algo_unroll_compare}. Based on the diagram, it is clear that the algorithms are just shifted versions of one another. If we initialize \cref{algo6} using $\xi_2^0 = x_1^1$ and $\xi_3^0 = x_3^0-x_1^1$, then it will make the same oracle calls as \cref{algo5}, but with a time shift.

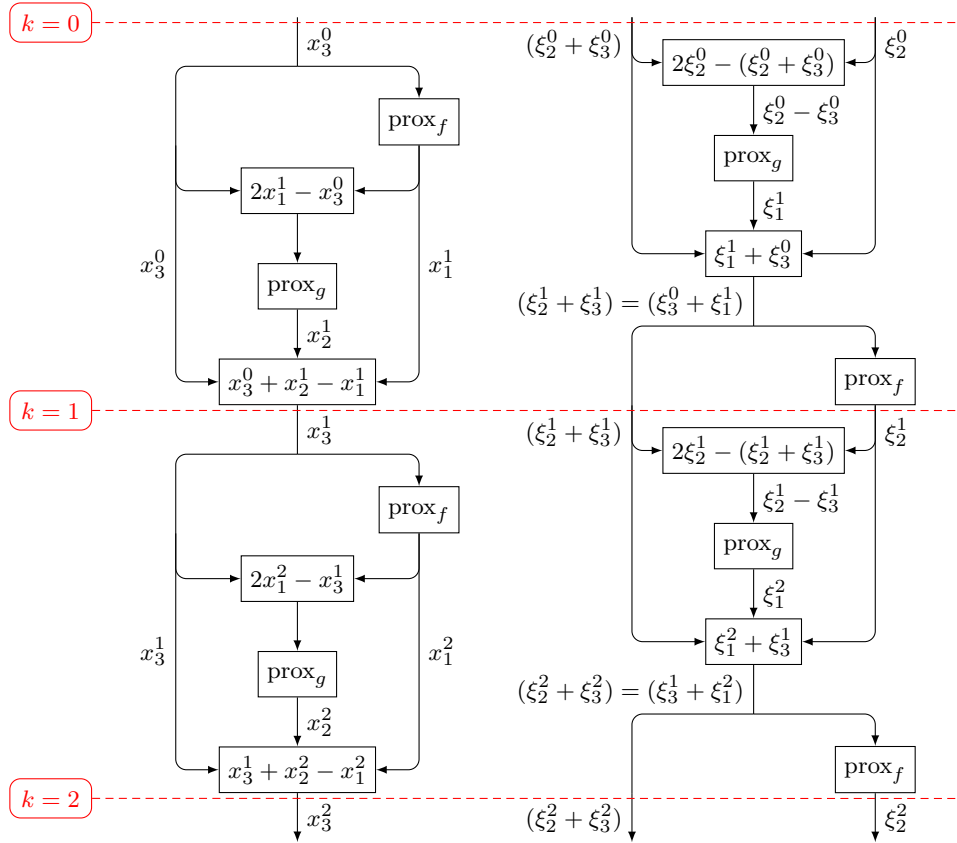
\begin{figure}[ht]
	\centering
	\begin{tikzpicture}[font=\small]
		\def\dy{1.6} % width of loops
		\def\db{0.65} % straight segment lengths
		\def\df{1.7} % dist between bottom of prox_f and above node
		\def\dn{0.6} % dist between bottom of prox_f and node below
		\tikzstyle{arr}=[>=latex,auto,->,rounded corners]
		\tikzstyle{timetag}=[draw,left,red,rounded corners,inner sep=1.5mm]

		% Algorithm 3.5
		\begin{scope}
			\coordinate (O) at (0,0);
			\coordinate (A) at (0,-\db);
			\coordinate (B) at (-\dy,-\df);
			\node[draw,anchor=south] (Pf) at (\dy,-\df) {$\text{prox}_f^{\vphantom{a}}$};
			\node[draw] (N1) at (0,-\df-\dn) {$\vphantom{\xi_1^1}2x_1^1-x_3^0$};
			\node[draw, below=\db of N1] (Pg) {$\text{prox}_g^{\vphantom{a}}$};
			\node[draw, below=\db of Pg] (N2) {$\vphantom{\xi_1^1}x_3^0+x_2^1-x_1^1$};
			\draw[auto] (0,0) -- node{$x_3^0$} (A);
			\draw[arr] (A) -| (B) |- node[swap,pos=0.25]{$x_3^0$} (N2);
			\draw[arr] (A) -| (Pf);
			\draw[arr] (B) |- (N1);
			\draw[arr] (Pf) |- (N1);
			\draw[arr] (N1) -- (Pg);
			\draw[arr] (Pg) -- node{$x_2^1$} (N2);
			\draw[arr] (Pf) |- node[pos=0.25]{$x_1^1$} (N2);
			\coordinate (R11) at (0,0);
		\end{scope}
		\begin{scope}[shift={(N2.south)}]
			\coordinate (A) at (0,-\db);
			\coordinate (B) at (-\dy,-\df);
			\node[draw,anchor=south] (Pf) at (\dy,-\df) {$\text{prox}_f^{\vphantom{a}}$};
			\node[draw] (N1) at (0,-\df-\dn) {$\vphantom{\xi_1^1}2x_1^2-x_3^1$};
			\node[draw, below=\db of N1] (Pg) {$\text{prox}_g^{\vphantom{a}}$};
			\node[draw, below=\db of Pg] (N2) {$\vphantom{\xi_1^1}x_3^1+x_2^2-x_1^2$};
			\draw[auto] (0,0) -- node{$x_3^1$} (A);
			\draw[arr] (A) -| (B) |- node[swap,pos=0.25]{$x_3^1$} (N2);
			\draw[arr] (A) -| (Pf);
			\draw[arr] (B) |- (N1);
			\draw[arr] (Pf) |- (N1);
			\draw[arr] (N1) -- (Pg);
			\draw[arr] (Pg) -- node{$x_2^2$} (N2);
			\draw[arr] (Pf) |- node[pos=0.25]{$x_1^2$} (N2);
			\coordinate (R21) at (0,0);
		\end{scope}
		\begin{scope}[shift={(N2.south)}]
			\coordinate (A) at (0,-\db);
			\draw[arr] (0,0) -- node{$x_3^2$} (A);
			\coordinate (R31) at (0,0);
		\end{scope}

		% Algorithm 3.6
		\begin{scope}[shift={($(O) + (6,0)$)}]
			\coordinate (B) at (-\dy,0);
			\coordinate (Pf) at (\dy,0);
			\node[draw] (N1) at (0,-\dn) {$2\xi_2^0 - (\xi_2^0+\xi_3^0)$};
			\node[draw, below=\db of N1] (Pg) {$\text{prox}_g^{\vphantom{a}}$};
			\node[draw, below=\db of Pg] (N2) {$\xi_1^1+\xi_3^0$};
			\draw[arr] (B) |- (N2);
			\draw[arr] (B) |- node[swap,pos=0.3]{$(\xi_2^0+\xi_3^0)$} (N1);
			\draw[arr] (Pf) |- node[pos=0.3]{$\xi_2^0$} (N1);
			\draw[arr] (N1) -- node{$\xi_2^0-\xi_3^0$} (Pg);
			\draw[arr] (Pg) -- node{$\xi_1^1$} (N2);
			\draw[arr] (Pf) |- (N2);
			\coordinate[below=\db of N2.south] (A);
			\draw[auto] (N2) -- node[swap]{$(\xi_2^1+\xi_3^1)=(\xi_3^0+\xi_1^1)$} (A);
			\node[draw,anchor=south] (Pf) at ($(N2.south) + (\dy,-\df)$) {$\text{prox}_f^{\vphantom{a}}$};
			\coordinate (q) at (-\dy,0);
			\draw[auto,rounded corners] (A) -| (q |- Pf.south) coordinate (g);
			\draw[arr] (A) -| (Pf);
			\coordinate (g) at (N2 |- Pf.south);
			\coordinate (R12) at (0,0);
		\end{scope}
		\begin{scope}[shift={(g)}]
			\coordinate (B) at (-\dy,0);
			\coordinate (Pf) at (\dy,0);
			\node[draw] (N1) at (0,-\dn) {$2\xi_2^1 - (\xi_2^1+\xi_3^1)$};
			\node[draw, below=\db of N1] (Pg) {$\text{prox}_g^{\vphantom{a}}$};
			\node[draw, below=\db of Pg] (N2) {$\xi_1^2+\xi_3^1$};
			\draw[arr] (B) |- (N2);
			\draw[arr] (B) |- node[swap,pos=0.3]{$(\xi_2^1+\xi_3^1)$} (N1);
			\draw[arr] (Pf) |- node[pos=0.3]{$\xi_2^1$} (N1);
			\draw[arr] (N1) -- node{$\xi_2^1-\xi_3^1$} (Pg);
			\draw[arr] (Pg) -- node{$\xi_1^2$} (N2);
			\draw[arr] (Pf) |- (N2);
			\coordinate[below=\db of N2.south] (A);
			\draw[auto] (N2) -- node[swap]{$(\xi_2^2+\xi_3^2)=(\xi_3^1+\xi_1^2)$} (A);
			\node[draw,anchor=south] (Pf) at ($(N2.south) + (\dy,-\df)$) {$\text{prox}_f^{\vphantom{a}}$};
			\coordinate (q) at (-\dy,0);
			\draw[auto,rounded corners] (A) -| (q |- Pf.south) coordinate (g);
			\draw[arr] (A) -| (Pf);
			\coordinate (g) at (N2 |- Pf.south);
			\coordinate (R22) at (0,0);
		\end{scope}
		\begin{scope}[shift={(g)}]
			\coordinate (B) at (-\dy,0);
			\coordinate (Pf) at (\dy,0);
			\draw[arr] (B) -- node[swap]{$(\xi_2^2+\xi_3^2)$} +(0,-\db);
			\draw[arr] (Pf) -- node{$\xi_2^2$} +(0,-\db);
			\coordinate (R32) at (0,0);
		\end{scope}

		\def\ds{0.07} % shift down of guidelines for legibility
		\def\dw{2.7} % width of guides past centerline

		\draw[red,densely dashed] ($(R11)+(-\dw,-\ds)$) coordinate (t0) -- ($(R12)+(\dw,-\ds)$);
		\draw[red,densely dashed] ($(R21)+(-\dw,-\ds)$) coordinate (t1) -- ($(R22)+(\dw,-\ds)$);
		\draw[red,densely dashed] ($(R31)+(-\dw,-\ds)$) coordinate (t2) -- ($(R32)+(\dw,-\ds)$);
		\node[timetag] at (t0) {$k=0$};
		\node[timetag] at (t1) {$k=1$};
		\node[timetag] at (t2) {$k=2$};
	\end{tikzpicture}
	\caption{Block diagrams representing \cref{algo5} (left) and \cref{algo6} (right). These algorithms are \emph{shift equivalent} because when suitably initialized, they make the same calls to the oracles, albeit with a time shift. The updates are exactly the same for both algorithms, but using transformed variables.}
	\label{fig:algo_unroll_compare}
\end{figure}

This example motivates us to define \emph{shift equivalence}. As with oracle equivalence, we ask whether the algorithms are indistinguishable from the point of view of the oracle for suitably chosen input channel delays and state initializations. Before we define shift equivalence, we will formalize the notion of shifting.

Shifting (delaying) a semi-infinite sequence $(y^0,y^1,\dots)$ by $m$ time steps corresponds to multiplication of its $z$-transform by $z^{-m}$:
\begin{align*}
	\mathcal{Z}\Bigl[(y^0,y^1,y^2,\dots)\Bigr]
&=
\sum_{k=0}^\infty y^k z^{-k}
= \hat y(z), \quad\text{and} \\
\mathcal{Z}\Bigl[(\underbrace{0,0,\dots,0}_{\textup{$m$ times}},y^0,y^1,y^2,\dots)\Bigr]
&=
\sum_{k=0}^\infty y^{k}z^{-m-k}
= z^{-m} \hat y(z).
\end{align*}
Without loss of generality, we can assume $\phi(0)=0$,\footnote{If $\phi(0)\neq 0$, we can redefine the shift operation to pad the first $m$ entries with $\phi(0)$ instead of $0$, which will ensure that $\phi$ commutes with the shift operation. Alternatively, we can re-center the algorithm states to be measured with respect to a fixed point of the dynamics, as is standard in control theory \cite[\S4.2]{michalowsky2021robust}.} and since the oracle $\phi$ applies element-wise to each $y^k$, the oracle $\phi$ commutes with the shift operation.
We can represent this relationship by a commutative diagram; see \cref{fig:commutative_diagram}.

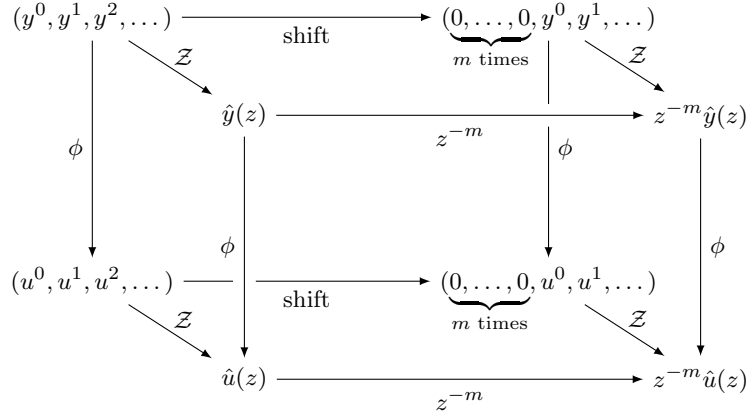
\begin{figure}[ht!]
	\centering
	\begin{tikzpicture}[font=\small,auto]
		\tikzstyle{bx} = []
		\tikzstyle{arr} = [>=latex,->]
		\def\dx{6}
		\def\dy{3.5}
		
		\begin{scope}
		\node[bx] (a) {$(y^0,y^1,y^2,\dots)$};
		\node[bx] (b) at (\dx,0) {\phantom{$(0,\dots,0,y^0,y^1,\dots)$}};
		\node[bx,anchor=north west] (bb) at (b.north west) {$(\underbrace{0,\dots,0}_{\textup{$m$ times}},y^0,y^1,\dots)$};
		\node[bx] (c) at (0,-\dy) {$(u^0,u^1,u^2,\dots)$};
		\node[bx] (d) at (b |- c) {\phantom{$(0,\dots,0,u^0,u^1,\dots)$}};
		\node[bx,anchor=north west] (dd) at (d.north west) {$(\underbrace{0,\dots,0}_{\textup{$m$ times}},u^0,u^1,\dots)$};
		\draw[arr] (a) -- node[below]{shift} (b);
		\draw[arr] (c) -- node[below]{shift} (d);
		\draw[arr] (a) -- node[left]{$\phi$} (c);
		\draw[arr] (b) -- node[right]{$\phi$} (d);
		\end{scope}
		
		\begin{scope}[shift={(2,-1.3)}]
			\node[bx] (a2) {$\hat y(z)$};
			\node[bx] at (\dx,0) (b2) {$z^{-m} \hat y(z)$};
			\node[bx] at (0,-\dy) (c2) {$\hat u(z)$};
			\node[bx] (d2) at (b2 |- c2) {$z^{-m}\hat u(z)$};
			\draw[white,line width=3mm] (a2) -- (b2);
			\draw[arr] (a2) -- node[below]{$z^{-m}$} (b2);
			\draw[arr] (c2) -- node[below]{$z^{-m}$} (d2);
			\draw[white,line width=3mm] (a2) -- (c2);
			\draw[arr] (a2) -- node[left]{$\phi$} (c2);
			\draw[arr] (b2) -- node[right]{$\phi$} (d2);
		\end{scope}

		\draw[arr] (a) -- node[inner sep=1pt]{$\mathcal{Z}$} (a2);
		\draw[arr] (b) -- node[inner sep=1pt]{$\mathcal{Z}$} (b2);
		\draw[arr] (c) -- node[inner sep=1pt]{$\mathcal{Z}$} (c2);
		\draw[arr] (d) -- node[inner sep=1pt]{$\mathcal{Z}$} (d2);
	\end{tikzpicture}
	\caption{Commutative diagram visualizing that the shift (delay) operation commutes with the application of the oracle $\phi$. The foreground shows the $z$-transformed versions of the signals, where the shift becomes multiplication by a power of $z^{-1}$.}
	\label{fig:commutative_diagram}
\end{figure}

For a vector-valued signal, we can delay each component by a different amount. This motivates the definition of the \emph{multi-shift}.
\medskip
\begin{definition}[multi-shift]\label{delay}
	We define the \emph{multi-shift} $\hat \Delta_m(z)$ for nonnegative integers $m \defeq (m_1,\dots,m_p)$ as the transfer function
	\[
		\hat \Delta_m(z)
		\defeq \bmat{z^{-m_1} & & 0 \\ & \ddots & \\
		0 & & z^{-m_p}}
	\]
\end{definition}
\medskip

For example, consider the semi-infinite sequence $(y^0,y^1,y^2,\dots)$, where each $y^k$ is partitioned into blocks $y^k_1,y^k_2,y^k_3$, where the $y^k_i$ are the same size for all $k$. Then,
	\[
		\hat\Delta_{(1,0,2)}(z) \hat y(z) = \mathcal{Z}\left[\left(
			\bmat{0    \\y_2^0\\0    },
			\bmat{y_1^0\\y_2^1\\0    },
			\bmat{y_1^1\\y_2^2\\y_3^0},
			\bmat{y_1^2\\y_2^3\\y_3^1},
			\dots,
			\bmat{y_1^{k-1}\\y_2^{k\phantom{+0}}\\y_3^{k-2}},
			\dots
			\right)\right].
	\] 
The multi-shift also commutes with any time-invariant oracle $\Phi = (\phi_1,\dots,\phi_p)$. In terms of $z$-transforms, $\Phi(\hat y) = \hat \Delta_m^{-1} \Phi(\hat \Delta_m \hat y)$. By rearranging the block diagram, we can move the multi-shifts from the oracle to the algorithm; see \cref{fig_shift_transformation}.

\begin{figure}[ht]
	\centering
	\begin{tikzpicture}[font=\small]
		\tikzstyle{blk}=[minimum width=8mm,minimum height=7mm,draw]
		\tikzstyle{arr}=[>=latex,auto,->,rounded corners]
		
		\begin{scope}
		\node[blk] at (0, 0) (algo) {$\hat H$};
		\node[blk,below= 0.3 of algo] (oracle) {$\Phi$};
		\coordinate (t1) at ($(algo.east) + (0.5,0)$);
		\coordinate (t2) at ($(algo.west) + (-0.5,0)$);
		\draw[arr]  (algo) -- (t1) |- (oracle);
		\draw[arr]  (oracle) -| (t2) -- (algo);
		\end{scope}

		\begin{scope}[shift={(3,0)}]
			\node[blk] at (0, 0) (algo) {$\hat H$};
			\node[blk,below= 0.3 of algo] (oracle) {$\hat \Delta_m^{-1} \Phi \hat\Delta_m$};
			\coordinate (t1) at ($(oracle.east) + (0.5,0)$);
			\coordinate (t2) at ($(oracle.west) + (-0.5,0)$);
			\draw[arr]  (algo) -| (t1) -- (oracle);
			\draw[arr]  (oracle) -- (t2) |- (algo);
		\end{scope}

		\begin{scope}[shift={(6.5,0)}]
			\node[blk] at (0, 0) (algo) {$\hat H$};
			\node[blk] at (0,-1.5) (oracle) {$\Phi$};
			\node[blk] at (1.1,-0.75) (D1) {$\hat\Delta_m$};
			\node[blk] at (-1.1,-0.75) (D2) {$\hat\Delta_m^{-1}$};
			\draw[arr]  (algo) -| (D1);
			\draw[arr]  (D1) |- (oracle);
			\draw[arr]  (oracle) -| (D2);
			\draw[arr]  (D2) |- (algo);
		\end{scope}

		\begin{scope}[shift={(10,0)}]
			\node[blk] at (0, 0) (algo) {$\hat\Delta_m\hat H\hat\Delta_m^{-1}$};
			\node[blk,below= 0.3 of algo] (oracle) {$\Phi$};
			\coordinate (t1) at ($(algo.east) + (0.5,0)$);
			\coordinate (t2) at ($(algo.west) + (-0.5,0)$);
			\draw[arr]  (algo) -- (t1) |- (oracle);
			\draw[arr]  (oracle) -| (t2) -- (algo);
		\end{scope}

	\end{tikzpicture}
	\caption{Equivalent block diagram representing shift equivalence. We use the fact that the oracle $\Phi$ commutes with any multi-shift $\hat\Delta_m$. However, $\hat\Delta_m$ need not commute with $\hat H$, which means equivalent algorithms can have different transfer functions.}
	\label{fig_shift_transformation}
\end{figure}
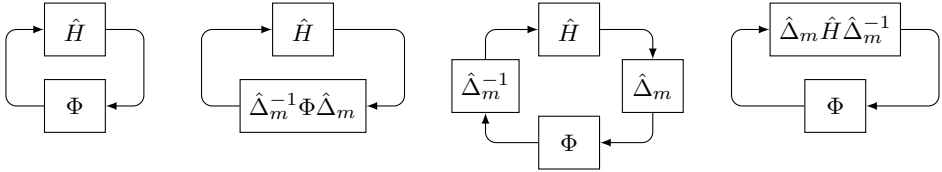

The transformation in \cref{fig_shift_transformation} shows that from the point of view of the oracle $\Phi$, the algorithms $\hat H$ and $\hat \Delta_m \hat H \hat \Delta_m^{-1}$ are indistinguishable. This motivates our definition of \emph{shift equivalence}.

\medskip
\begin{definition}[shift equivalence]\label{def2}
	Suppose we are given two LTI algorithms that each use the same $p$ oracles and have transfer functions $\hat H_1$ and $\hat H_2$, respectively. We say they are \emph{shift-equivalent} and write $\hat H_1 \sim \hat H_2$ if there exists a multi-shift $\hat \Delta_m$ such that
		\[
			\hat H_1 = \hat\Delta_m \hat H_2 \hat\Delta_m^{-1}.
		\]
\end{definition}
\medskip

Our choice of the word \emph{equivalence} is justified by the fact that shift equivalence is an \emph{equivalence relation}.

\medskip
\begin{lemma}\label{lem:shift_equiv}
	Shift equivalence, as defined in \cref{def2}, is an \emph{equivalence relation}. That is, it satisfies the properties of reflexivity, symmetry, and transitivity.
\end{lemma}

\begin{proof} See \cref{app:equivalence1}.
\end{proof}

\medskip
\begin{remark}
	Oracle equivalence is a special case of shift equivalence (with $\hat \Delta_m = I$). Moreover,
	shift equivalence reduces to oracle equivalence when $p=1$ (a single oracle). In this case, the transfer functions and multi-shifts are scalars rather than matrices, so they trivially commute:
	$z^{-m}\hat H_1 = \hat H_2 z^{-m} \iff \hat H_1 = \hat H_2$.
\end{remark}

\titleparagraph{Efficient enumeration of shift-equivalent algorithms}

Given an algorithm $\hat H$ using oracles $\Phi=(\phi_1,\dots,\phi_p)$, how can we generate all possible shift-equivalent algorithms? In other words, what are the possible transfer functions $\hat H'$ and multi-shifts $\hat \Delta_m$ such that
\[
\hat H'(z) =
\bmat{z^{-m_1} & & 0 \\ & \ddots & \\
		0 & & z^{-m_p}} \hat H \bmat{z^{m_1} & & 0 \\ & \ddots & \\
		0 & & z^{m_p}}.
\]
Since $\hat H'$ must be a proper transfer function (or strictly proper, depending on
whether we require explicit implementations; see \cref{app:construction}), each $\hat H_{ij}(z) z^{m_j-m_i}$ must also be proper. Here, \emph{proper} means that each entry is a rational function whose
numerator degree is no larger than its denominator degree, while \emph{strictly proper}
means the numerator degree is strictly smaller. To determine the set of possible
$\hat H'$, let $r_{ij}$ be the relative degree of $\hat H_{ij}$.
That is, write $\hat H_{ij}(z) = N_{ij}(z) / D_{ij}(z)$ (ratio of polynomials), and
\[
r_{ij} \defeq \begin{cases}
+\infty & \text{if }\hat H_{ij}(z) = 0 \\
\deg(D_{ij}) -\deg(N_{ij}) & \text{otherwise}
\end{cases}
\]
Then, properness of $\hat H'$ amounts to finding $m_i$ such that
\begin{equation}\label{eq:bounds}
-r_{ij} \leq m_i - m_j \leq r_{ji}\qquad\text{for all }i \neq j.
\end{equation}
Since the set of feasible $m_i$ is translation-invariant, we can normalize each solution so that $\min_i m_i = 0$, and each distinct solution $\{m_i\}$ will correspond to a distinct $\hat H'$ that is shift-equivalent to $\hat H$. For an example of how we can enumerate solutions, see the primal-dual three-operator splitting example in \cref{app-shift}.

\titleparagraph{Efficient determination of shift equivalence}

Given two algorithms $\hat H$ and $\hat H'$ that use the same oracles $(\phi_1,\dots,\phi_p)$, we can efficiently check whether these algorithms are shift-equivalent by carrying out the following steps.
\begin{enumerate}
	\item Check to make sure $\hat H_{ii} = H'_{ii}$ for all $i$ (the diagonal entries must always match). Otherwise, they are not equivalent.
	\item Check to make sure that for all $i \neq j$, either $\hat H_{ij} = \hat H'_{ij} = 0$, or $H_{ij} \neq 0$ and $\hat H'_{ij} \neq 0$. In other words, $\hat H$ and $\hat H'$ must have matching sparsity patterns. Otherwise, they are not equivalent.
	\item For all $i \neq j$ such that $\hat H_{ij} \neq 0$, check to make sure that $H'_{ij}(z) / \hat H_{ij}(z) = z^{b_{ij}}$ for some integers $b_{ij}$. In other words, corresponding entries of the algorithm's transfer functions must by related by multiplication by a power of $z$.
	\item Consider the set of linear equations $b_{ij} = m_j - m_i$ for all $i\neq j$ such that $\hat H_{ij} \neq 0$. Write the corresponding linear equations compactly as $M^\top m = b$. If this system of equations has a solution, then $\hat H \sim \hat H'$. Otherwise, they are not equivalent. Note that if a solution exists, we can always find a solution with integer $m$, since $M$ is an incidence matrix and therefore is totally unimodular.
\end{enumerate}

\subsection{Examples of shift equivalence}\label{app-shift}

\titleparagraph{Douglas--Rachford splitting and ADMM}

As computed in \cref{ex:shift}, the transfer functions for \cref{algo5,algo6} are given by
$\hat H_{\ref{algo5}} = \bmat{
		\frac{-1}{z-1} & \frac{1}{z-1} \\
		\frac{2 z-1}{z-1} & \frac{-1}{z-1}}$
and
$\hat H_{\ref{algo6}} = \bmat{
		\frac{-1}{z-1} & \frac{z}{z-1} \\
		\frac{2 z-1}{z(z-1)} & \frac{-1}{z-1}}$.
We see that they are related via
\(
\hat H_{\ref{algo5}} = \bmat{z^{-1} & 0 \\ 0 & 1}\hat H_{\ref{algo6}}\bmat{z & 0 \\ 0 & 1}
\).
Therefore, $\hat H_{\ref{algo5}}\sim \hat H_{\ref{algo6}}$.

\titleparagraph{Primal-dual three-operator splitting}

For our next example, we consider algorithms for solving the optimization problem
\begin{equation*}
\ba{ll}
\mbox{minimize} & f(x) + g(Ax) + h(x)
\ea
\end{equation*}
using the oracles $\text{prox}_{\tau f}$, $\text{prox}_{\sigma g^*}$, and $\nabla h$. Only recently have methods been proposed to solve this problem. Examples include the Condat--V\~{u} algorithm independently proposed by Condat and V\~{u} \cite{Condat,Vu}, the primal-dual three-operator (PD3O) algorithm  \cite{PD3O}, and the primal-dual Davis--Yin (PDDY) algorithm \cite{PDDY}. See \cite{jiang2023bregman} and references therein for a recent survey on this problem.
To illustrate our approach, we will focus on the primal-dual three-operator algorithm (PD3O)~\cite{PD3O}, shown below.

\vspace{-1em}
\begin{algorithm}[H]
	\centering
	\captionsetup{font=scriptsize}
	\caption{Primal-dual three-operator splitting (PD3O)}
	\label{algoPD3O}
	\scriptsize
	\begin{algorithmic}
		\For{$k=0, 1, 2,\ldots$}
		\State{$x^{k} = \prox_{\tau f}(z^k)$}
		\State{$s^{k+1} = \prox_{\sigma g^*}\Bigl((I-\tau\sigma AA^\top )s^k + \sigma A\bigl(2x^k-z^k-\tau \nabla h(x^k)\bigr)\Bigr)$}
		\State{$z^{k+1} = x^k - \tau \nabla h(x^k)-\tau A^\top  s^{k+1}$}
		\EndFor
	\end{algorithmic}
\end{algorithm}
\vspace{-1em}
\noindent The state-space realization and transfer function for PD3O is\footnote{We have removed identity matrices from the transfer function to simplify exposition. For example, entries in $\hat H_{\ref{algoPD3O}}(z)$ that read $\frac{1}{z}$ should be replaced by $\frac{1}{z} I$.}
\[
\hat H_{\ref{algoPD3O}}(z) = \left[\begin{array}{cc|ccc}
	0 & 0 & 0 & I & 0 \\
	0 & 0 & I & -\tau A^\top  & -\tau I \\ \hline
	0 & I & 0 & 0 & 0 \\
	I-\tau \sigma AA^\top  & -\sigma A & 2\sigma A & 0 & -\tau\sigma A \\
	0 & 0 & I & 0 & 0
\end{array}\right]
= \bmat{\frac{1}{z} & \frac{-\tau A^\top  }{z} & \frac{-\tau }{z} \\
	 \frac{\sigma  (2 z-1)A}{z} & \frac{1}{z} & \frac{-\sigma  \tau  (z-1) A}{z} \\
	 1 & 0 & 0 }
\]
In \cite{PD3O}, the authors show a reformulation of PD3O and state that it was obtained by changing the order of the variables and substituting $\bar x^k = 2x^k -z^k - \tau \nabla h(x^k)-\tau A^\top  s^k$. After these changes, the reformulation is given by \cref{algoPD3Ob} below.

\vspace{-1em}
\noindent
\hfil
\begin{minipage}{0.56\linewidth}
\begin{algorithm}[H]
	\centering
	\captionsetup{font=scriptsize}
	\caption{Reformulation of PD3O}
	\label{algoPD3Ob}
	\scriptsize
	\begin{algorithmic}
		\For{$k=0, 1, 2,\ldots$}
		\State{$s^{k+1} = \prox_{\sigma g^*}(s^k + \sigma A\bar x^k)$}
		\State{$x^{k+1} = \prox_{\tau f}(x^k - \tau\nabla h(x^k)-\tau A^\top  s^{k+1})$}
		\State{$\bar x^{k+1} = 2x^{k+1} -x^k + \tau \nabla h(x^k)-\tau \nabla h(x^{k+1})$}
		\EndFor
	\end{algorithmic}
\end{algorithm}
\end{minipage}
\hfil
\begin{minipage}{0.42\linewidth}
	\[
		\hat H_{\ref{algoPD3Ob}} =
		\addtolength{\arraycolsep}{-1mm}
		\bmat{\frac{1}{z} & -\tau A^\top  & \frac{-\tau }{z} \\
		\frac{\sigma  (2 z-1)A}{z^2} & \frac{1}{z} & \frac{-\sigma  \tau  (z-1) A}{z^2} \\
		1 & 0 & 0 }
	\]
\end{minipage}
\hfil
\vspace{1em}

We can obtain the equivalence between \cref{algoPD3O,algoPD3Ob} immediately. Indeed,  \cref{algoPD3O,algoPD3Ob} are shift equivalent because
\[
	\hat H_{\ref{algoPD3Ob}}
	=
	\bmat{1 & 0 & 0\\0 & z^{-1} & 0\\0 & 0 & 1}
	\hat H_{\ref{algoPD3O}}
\bmat{1 & 0 & 0\\0 & z & 0\\0 & 0 & 1}.
\]
This is not the only possible shift-equivalence transformation. Applying the method outlined in \cref{eq:bounds}, the relative degree matrix of $\hat H_{\ref{algoPD3O}}$ is given by
\[
[r_{ij}] = \bmat{1 & 1 & 1 \\ 0 & 1 & 0 \\ 0 & \infty & \infty}.
\]
We seek nonnegative integers $(m_1,m_2,m_3)$ normalized so that $\min_i m_i=0$ satisfying
\begin{align*}
0 &\leq m_2 - m_1 \leq 1,  &
0 &\leq m_3 - m_1 \leq 1,  &
0 &\leq m_2 - m_3 \leq \infty. 
\end{align*}
\cref{algoPD3O} corresponds to the trivial solution $(0,0,0)$, and \cref{algoPD3Ob} corresponds to $(0,1,0)$. By inspection, we see there are three solutions. The third solution is $(0,1,1)$ and it corresponds to the new algorithm
\[
	\hat H_{\ref{algoPD3Oc}}
	=	
	\bmat{1 & 0 & 0\\0 & z^{-1} & 0\\0 & 0 & z^{-1}}
	\hat H_{\ref{algoPD3O}}
	\bmat{1 & 0 & 0\\0 & z & 0\\0 & 0 & z}.
\]
One possible realization of \cref{algoPD3Oc} is given below.

\vspace{-1em}
\noindent
\hfil
\begin{minipage}{0.53\linewidth}
\begin{algorithm}[H]
	\centering
	\captionsetup{font=scriptsize}
	\caption{Another reformulation of PD3O}
	\label{algoPD3Oc}
	\scriptsize
	\begin{algorithmic}
		\For{$k=0, 1, 2,\ldots$}
		\State{$y^k = \prox_{\sigma g^*}\bigl(\sigma A(2w^k-\tau\nabla h(w^k))-x^k\bigr)$}
		\State{$x^{k+1} = \sigma A\bigl(w^k - \tau\nabla h(w^k)\bigr) - y^k $}
		\State{$w^{k+1} = \prox_{\tau f}(x^{k+1})$}
		\EndFor
	\end{algorithmic}
\end{algorithm}
\end{minipage}
\hfil
\begin{minipage}{0.45\linewidth}
	\[
		\hat H_{\ref{algoPD3Oc}} =
		\addtolength{\arraycolsep}{-1mm}
		\bmat{ \frac{1}{z} & -\tau A^\top   & -\tau  \\
		 \frac{\sigma  (2 z-1)A}{z^2} & \frac{1}{z} & \frac{-\sigma  \tau  (z-1)A}{z} \\
		 \frac{1}{z} & 0 & 0 }
	\]
\end{minipage}
\hfil
\vspace{1em}

\noindent We stress that these equivalences are tedious to work out by hand, and since there are now three oracles, the equivalences are far from obvious. For example, here is another way to realize \cref{algoPD3Ob}. This time, the transfer functions are the same, so \cref{algoPD3Ob,algoPD3Od} are oracle-equivalent.

\vspace{-1em}
\noindent
\hfil
\begin{minipage}{0.53\linewidth}
\begin{algorithm}[H]
	\centering
	\captionsetup{font=scriptsize}
	\caption{Yet another reformulation of PD3O}
	\label{algoPD3Od}
	\scriptsize
	\begin{algorithmic}
		\For{$k=0, 1, 2,\ldots$}
		\State{$y^k = \prox_{\sigma g^*}(2\sigma A w^k - x^k)$}
		\State{$z^k = \prox_{\tau f}(w^k - \tau A^\top  y^k)$}
		\State{$x^{k+1} = \sigma A\bigl( w^k - \tau \nabla h(z^k) \bigr) - y^k$}
		\State{$w^{k+1} = z^k - \tau\nabla h(z^k)$}
		\EndFor
	\end{algorithmic}
\end{algorithm}
\end{minipage}
\hfil
\begin{minipage}{0.45\linewidth}
	\[
		\hat H_{\ref{algoPD3Od}} =
		\addtolength{\arraycolsep}{-1mm}
		\bmat{\frac{1}{z} & -\tau A^\top  & \frac{-\tau }{z} \\
		\frac{\sigma  (2 z-1)A}{z^2} & \frac{1}{z} & \frac{-\sigma  \tau  (z-1) A}{z^2} \\
		1 & 0 & 0 }
	\]
\end{minipage}
\hfil
\vspace{1em}

%%%%%%%%%%%%%%%%%%%%%%%%%%%%%%%%%%%%%%%%%%%%%%%%%%%%%%%%%%%%%%%%%%%%%%%%%%%%%%%%%%%%%%%%%%%

\section{LFT equivalence}\label{lft-equ}

In \cref{oracle-equ,shift-equ}, we considered equivalence between algorithms that use the same oracles. In this section, we consider equivalence between algorithms that use \emph{different but related} oracles.

In convex optimization,
algorithm conjugation naturally relates some oracles to others \cite{ryu2016primer}\cite[\S2]{ryuyinconvex}: % should be \cites here
for example, if
$(\partial f)(x) \defeq \left\{ g \;\mid\; f(y) \geq f(x) +g^\top (y-x)\text{ for all }y\right\}$ is the subdifferential of $f$,
$\prox_f(v) \defeq \argmin_x\bigl( f(x) + \frac{1}{2}\norm{x-v}^2\bigr)$ is the proximal operator of $f$, and 
$f^*(y) \defeq \sup_x \{x^\top y - f(x)\}$ is the Fenchel conjugate of $f$ \cite[\S3]{fenchel1953convex}, we have the following identities relating the different operators.
\begin{itemize}
	\item $y \in \partial f(x) \quad\iff\quad x \in \partial f^*(y)$, %(\partial f)^{-1} = \partial f^*$, and 
	\item $y = \prox_{tf}(x) \quad\iff\quad x \in y + t\partial f(y)$
	\item $x = \prox_{tf}(x) + t \prox_{\frac{1}{t}f^*}(\tfrac{1}{t}x)$\quad (Moreau's identity)
\end{itemize}
We can rewrite any algorithm in terms of different, also easily computable,
oracles using these identities.
Consider a simple example: we will obfuscate
the proximal gradient method (\cref{algo11} \cite[\S10]{doi:10.1137/1.9781611974997}\cite{doi:10.1137/080716542})
by rewriting it in terms of the conjugate of the original oracle $\prox_g$,
using Moreau's identity, as \cref{algo12}~\cite{moreau:hal-01867187}. These are the same as our motivating examples of \cref{algo11x,algo12x}.

\vspace{-1em}
\noindent
\hfil
\begin{minipage}{0.47\linewidth}
\begin{algorithm}[H]
	\centering
	\captionsetup{font=scriptsize}
	\caption{Proximal gradient}
	\label{algo11}
	\scriptsize
	\begin{algorithmic}
		\For{$k=0, 1, 2,\ldots$}
		\State{$y^k = x^k-t \nabla f(x^k)$}
		\State{$x^{k+1} = \prox_{tg}(y^k)\vphantom{\prox_{\frac{1}{t}g^*}(\frac{1}{t}y^k)}$}
		\EndFor
	\end{algorithmic}
\end{algorithm}
\end{minipage}
\hfil
\begin{minipage}{0.47\linewidth}
	\begin{algorithm}[H]
	\centering
	\captionsetup{font=scriptsize}
	\caption{Conjugate proximal gradient}
	\label{algo12}
	\scriptsize
	\begin{algorithmic}
		\For{$k=0, 1, 2,\ldots$}
		\State{$y^k = x^k - t\nabla f(x^k)$}
		\State{$x^{k+1} =  y^k -
			t\prox_{\frac{1}{t}g^*}(\frac{1}{t}y^k)$}
		\EndFor
	\end{algorithmic}
\end{algorithm}
\end{minipage}
\hfil
\vspace{1em}

\noindent We can also use the relationship between the proximal and subdifferential operators to obtain versions of \cref{algo11,algo12} that use subdifferentials instead.

\vspace{-1em}
\noindent
\hfil
\begin{minipage}{0.47\linewidth}
\begin{algorithm}[H]
	\centering
	\captionsetup{font=scriptsize}
	\caption{Subdifferential}
	\label{algo11b}
	\scriptsize
	\begin{algorithmic}
		\For{$k=0, 1, 2,\ldots$}
		\State{$y^k = x^k-t \nabla f(x^k)$}
		\State{$x^{k+1} \in y^k - t \partial g(\vphantom{\frac{1}{t}}x^{k+1})$}
		\EndFor
	\end{algorithmic}
\end{algorithm}
\end{minipage}
\hfil
\begin{minipage}{0.47\linewidth}
	\begin{algorithm}[H]
	\centering
	\captionsetup{font=scriptsize}
	\caption{Conjugate Subdifferential}
	\label{algo12b}
	\scriptsize
	\begin{algorithmic}
		\For{$k=0, 1, 2,\ldots$}
		\State{$y^k = x^k - t\nabla f(x^k)$}
		\State{$x^{k+1} \in  \partial g^*\bigl( \frac{1}{t}(y^k-x^{k+1})\bigr)$}
		\EndFor
	\end{algorithmic}
\end{algorithm}
\end{minipage}
\hfil
\vspace{1em}

Note that the update equations for \cref{algo11b,algo12b} involving subdifferentials are implicit.
The transfer functions for \crefrange{algo11}{algo12b} are shown below, along with their associated oracles. We use the symbol $\algcomp$ to show that an algorithm is used with a given set of oracles.
% the tfs that were here before were for the subdifferentials, not for the proxes.
\begin{equation}\label{eq:conj_tfs}
\begin{aligned}
	\text{Algo.~\ref{algo11}} &:
	\bmat{ 0 & \frac{1}{z} \\[1mm]
		-t & \frac{1}{z}}
		\algcomp (\nabla f, \prox_{tg}), &
	\text{Algo.~\ref{algo12}} &:
	\bmat { \frac{-t}{z-1} & \frac{-t}{z-1} \\[1mm]
		\frac{-z}{z-1} & \frac{-1}{z-1} }
		\algcomp (\nabla f, \prox_{\frac{1}{t}g^*}), \\
	\text{Algo.~\ref{algo11b}} &:
	\bmat{\frac{-t}{z-1} & \frac{-t}{z-1}  \\[1mm]
		\frac{-tz}{z-1} & \frac{-tz}{z-1} }
		\algcomp (\nabla f, \partial g), &
	\text{Algo.~\ref{algo12b}} &:
	\bmat{0 & \frac{1}{z}  \\[1mm]
		-1 & \frac{1-z}{tz} }
		\algcomp (\nabla f, \partial g^*).
\end{aligned}
\end{equation}
Although the transfer functions of the algorithms change when we rewrite the algorithm to call a different oracle, the sequence of states is preserved ($x^k$ and $y^k$ have the same values for all algorithms provided they are initialized the same way). This motivates us to define a general notion of equivalence that applies when two algorithms use different oracles that are related in a particular way.

\medskip
\begin{definition}[Operator graph] Given an oracle $\Phi:\mathcal{V}^p \to \mathcal{V}^p$, we define its \emph{graph} as the set of possible input-output pairs (in the $z$-domain). We adopt the linear algebraic notation $\mathcal{R}$ (range) overloaded as follows:
\[
\mathcal{R}\bmat{I \\ \Phi} \defeq
\left\{ \bmat{ \hat y \\ \Phi(\hat y)} \;\middle|\; \hat y = \mathcal{Z}[y^k],\text{ where } y^k \in \mathcal{V}^p \text{ for }k=0,1,\dots\right\}.
\]
Likewise, given an algorithm $\hat H$, we define its \emph{dual graph} as:
\[
\mathcal{R}\bmat{\hat H \\ I} \defeq
\left\{ \bmat{ \hat H \hat u \\ \hat u} \;\middle|\; \hat u = \mathcal{Z}[u^k],\text{ where } u^k \in \mathcal{V}^p \text{ for }k=0,1,\dots\right\}.
\]
\end{definition}

\medskip

\begin{definition}[Linearly equivalent oracles]\label{def:lin_related_oracles}
	Let $\Phi_1$ and $\Phi_2$ be oracles. We say that $\Phi_1$ is \emph{linearly equivalent} to $\Phi_2$ and we write $\Phi_1 \Msim \Phi_2$, if their graphs are related by an invertible linear transformation $\hat M$.
	In other words, $\Phi_1 \Msim \Phi_2$ if
	\[
	\mathcal{R}\bmat{ I \\ \Phi_1} =
	\hat M\,
	\mathcal{R}\bmat{I \\ \Phi_2}.
	\]
	This is equivalent to saying that:
	\begin{itemize}
		\item If $\hat u_1 = \Phi_1(\hat y_1)$, there exists $\hat y_2,\hat u_2$ such that $\hat u_2 = \Phi_2(\hat y_2)$ and $\bmat{\hat y_1 \\ \hat u_1} = \hat M \bmat{\hat y_2 \\ \hat u_2}$, and \item If $\hat u_2 = \Phi_2(\hat y_2)$, there exists $\hat y_1,\hat u_1$ such that $\hat u_1 = \Phi_1(\hat y_1)$ and $\bmat{\hat y_1 \\ \hat u_1} = \hat M \bmat{\hat y_2 \\ \hat u_2}$.
	\end{itemize}
	We will omit $\hat M$ and simply write $\Phi_1 \sim \Phi_2$ to mean that there exists some invertible $\hat M$ such that $\Phi_1 \Msim \Phi_2$. 
\end{definition}

\medskip

\noindent For example, $\prox_f  \sim \partial f$ because:
\begin{equation}\label{eq:example_prox_eqn}
\begin{cases}
x = y + t\partial f(y)\\
y = \prox_{tf}(x)
\end{cases}
\quad\iff\quad
\bmat{x \\ \prox_{tf}(x)} = \bmat{1 & t \\ 1 & 0}\bmat{y \\ \partial f(y)}.
\end{equation}
Since each $x$ corresponds to some $y$ and vice versa, it also holds for the $z$-transforms of arbitrary sequences $(x^0,x^1,\dots)$ and corresponding $(y^0,y^1,\dots)$ using the same $2\times 2$ matrix.

\medskip
\begin{proposition}(Special cases of linear relations)\label{prop:LFT_properties}
	\begin{enumerate}
		\itemsep=2mm
		\item \emph{Identity:} If $\phi_1 = \phi_2$, then $\phi_1 \Msim \phi_2$ with $\hat M = \sbmat{I & 0 \\ 0 & I}$.
				
		\item \emph{Commutation:} If $\phi(\hat C \hat y) = \hat C \phi(\hat y)$ for all $\hat y$, then $\phi\Msim \phi$ with $\hat M = \sbmat{\hat C & 0 \\ 0 & \hat C}$.
		
		\item \emph{Equivariance:} If $\phi_1(\hat A \hat y) = \hat B \phi_2(\hat y)$ for all $\hat y$, then $\phi_1 \Msim \phi_2$ with $\hat M = \sbmat{\hat A & 0 \\ 0 & \hat B}$.

		\item \emph{Concatenation:} If $\psi_i \overset{\hat M_i}{\sim} \phi_i$ with $\hat M_i = \sbmat{\hat P_i & \hat Q_i \\ \hat R_i & \hat S_i}$ for $i=1,\dots,p$, then\\
		$(\psi_1,\dots,\psi_p) \Mpsim (\phi_1,\dots,\phi_p)$ with
		$\hat M' = \sbmat{\diag(\hat P_i) & \diag(\hat Q_i) \\ \diag(\hat R_i) & \diag(\hat S_i)}$.
	\end{enumerate}
\end{proposition}
\medskip

When $\Phi_1 \Msim \Phi_2$ and these oracles are used with algorithms $\hat H_1$ and $\hat H_2$, respectively, we must have $\hat y_1 = \hat H_1 \hat u_1$ and $\hat y_2 = \hat H_2 \hat u_2$. Incorporating this with \cref{def:lin_related_oracles}, we can define a natural generalization of equivalence that holds in this setting.

\medskip
\begin{definition}[LFT equivalence]\label{def:LFT_equivalence}
	Consider $\hat H_1 \algcomp \Phi_1$ and $\hat H_2 \algcomp \Phi_2$, where $\Phi_1 \Msim \Phi_2$.
	We say the algorithms are \emph{LFT-equivalent} and write 
	$\hat H_1 \algcomp \Phi_1\Msim\hat H_2 \algcomp \Phi_2$, if
	\begin{equation}\label{eq:def_lft}
	\mathcal{R}\bmat{\hat H_1 \\ I} = \hat M \, 
	\mathcal{R}\bmat{\hat H_2 \\ I}.
	\end{equation}
	We justify the name ``LFT'' in \cref{rem:name_choice} and ``equivalence'' in \cref{rem:equivalence}.

	We omit $\hat M$ and simply write $\hat H_1 \algcomp \Phi_1 \sim \hat H_2 \algcomp \Phi_2$ when $\hat M$ is the same as that for which $\Phi_1 \Msim \Phi_2$, and therefore clear from context.
\end{definition}
\medskip

\begin{remark}\label{rem:equivalence}
	Linear equivalence $\Phi_1 \sim \Phi_2$ and LFT equivalence $\hat H_1 \algcomp \Phi_1 \sim \hat H_2 \algcomp \Phi_2$ satisfy \emph{reflexivity} and \emph{symmetry}, and \emph{transitivity}, so they are \emph{equivalence relations}.
\end{remark}
\medskip

Our main result of this section is an algebraic characterization of LFT equivalence between algorithms defined in \cref{def:LFT_equivalence}.

\medskip
\begin{theorem}[algebraic characterization of LFT equivalence]\label{thm:main_LFT_result}
	Suppose $\Phi_1 \Msim \Phi_2$. 
	Then $\hat H_1 \algcomp \Phi_1 \Msim \hat H_2 \algcomp \Phi_2$ if and only if
	\[
	\bmat{I & -\hat H_1} \hat M \bmat{\hat H_2 \\ I} = 0,
	\quad
	\text{or equivalently,}
	\quad
	\bmat{I & -\hat H_2} \hat M^{-1} \bmat{\hat H_1 \\ I} = 0.
	\]
\end{theorem}

\begin{proof}
See \cref{app:equivalence2}.
\end{proof}

We can apply \cref{thm:main_LFT_result} to solve for $\hat H_1$ in terms of $\hat H_2$ or vice versa.

\medskip
\begin{corollary}\label{cor:main_LFT_result}
	Consider the setting of \cref{thm:main_LFT_result} with $\hat M = \sbmat{\hat P & \hat Q \\ \hat R & \hat S}$. Then we have $\hat H_1(\hat R \hat H_2 + \hat S) = (\hat P \hat H_2 + \hat Q)$. In particular,
	\begin{equation}\label{eq:LFT_thm}
		\hat H_1 = (\hat P \hat H_2 + \hat Q) (\hat R \hat H_2 + \hat S)^{-1}
		\quad\text{and}\quad
		\hat H_2 = (- \hat H_1 \hat R + \hat P )^{-1} (\hat H_1 S - \hat Q).
	\end{equation}
\end{corollary}
\medskip

\begin{remark}\label{rem:name_choice}
	The relationships between $\hat H_1$ and $\hat H_2$ in \cref{eq:LFT_thm} are commonly called \emph{linear fractional transformations} (LFTs), which is why we chose the name \emph{LFT equivalence}.
\end{remark}
\medskip

The results above can also be derived by direct manipulation of the block diagram
as we demonstrated with shift equivalence in \cref{fig_shift_transformation}. 
In this case, the manipulation is a bit more involved; see \cref{fig_LFT_transformation}.

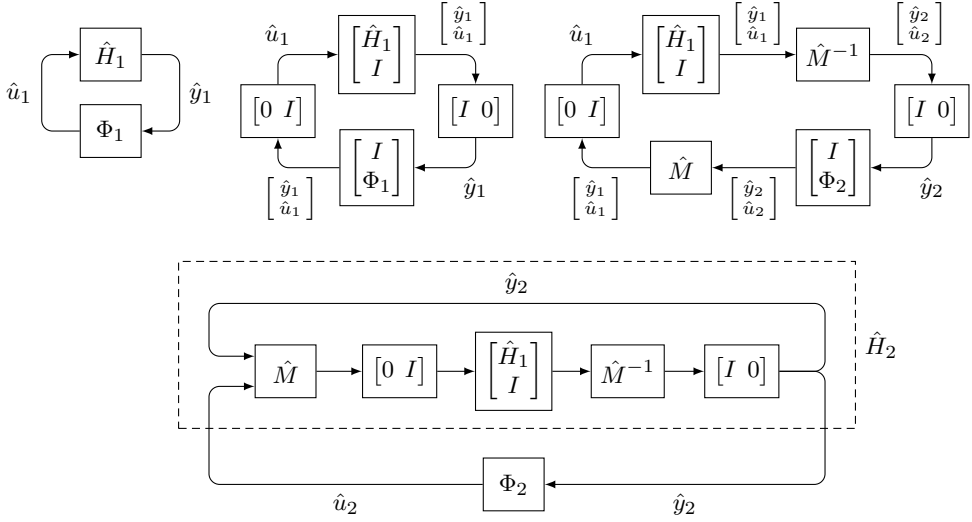
\begin{figure}[ht]
	\centering
	\begin{tikzpicture}[font=\small]
		\tikzstyle{blk}=[minimum width=8mm,minimum height=7mm,draw]
		\tikzstyle{arr}=[>=latex,auto,->,rounded corners]
		
		\begin{scope}
		\node[blk] at (0, 0) (algo) {$\hat H_1$};
		\node[blk,below= 0.3 of algo] (oracle) {$\Phi_1$};
		\coordinate (t1) at ($(algo.east) + (0.5,0)$);
		\coordinate (t2) at ($(algo.west) + (-0.5,0)$);
		\draw[arr]  (algo) -- (t1) |- node[pos=0.25]{$\hat y_1$} (oracle);
		\draw[arr]  (oracle) -| node[pos=0.75]{$\hat u_1$} (t2) -- (algo);
		\end{scope}

		\begin{scope}[shift={(3.5,0)}]
			\node[blk] at (0, 0) (algo) {$\bmat{\hat H_1\\I}$};
			\node[blk] at (0,-1.5) (oracle) {$\bmat{I \\ \Phi_1}$};
			\node[blk] at (1.3,-0.75) (D1) {$\bmat{I & 0}$};
			\node[blk] at (-1.3,-0.75) (D2)  {$\bmat{0 & I}$};
			\draw[arr]  (algo) -| node[pos=0.4]{$\sbmat{\hat y_1 \\ \hat u_1}$} (D1);
			\draw[arr]  (D1) |- node{$\hat y_1$} (oracle);
			\draw[arr]  (oracle) -| node[pos=0.4]{$\sbmat{\hat y_1 \\ \hat u_1}$} (D2);
			\draw[arr]  (D2) |- node{$\hat u_1$} (algo);
		\end{scope}

		\begin{scope}[shift={(7.5,0)}]
			\node[blk] at (0, 0) (algo) {$\bmat{\hat H_1\\I}$};
			\node[blk] at (2,0) (Mi) {$\hat M^{-1}$};
			\node[blk] at (2,-1.5) (oracle) {$\bmat{I \\ \Phi_2}$};
			\node[blk] at (0,-1.5) (M) {$\hat M$};
			\node[blk] at (3.3,-0.75) (D1) {$\bmat{I & 0}$};
			\node[blk] at (-1.3,-0.75) (D2)  {$\bmat{0 & I}$};
			\draw[arr]  (algo) -- node{$\sbmat{\hat y_1 \\ \hat u_1}$} (Mi);
			\draw[arr]  (Mi) -| node[pos=0.4]{$\sbmat{\hat y_2 \\ \hat u_2}$} (D1);
			\draw[arr]  (D1) |- node{$\hat y_2$} (oracle);
			\draw[arr]  (oracle) -- node{$\sbmat{\hat y_2 \\ \hat u_2}$} (M);
			\draw[arr]  (M) -| node[pos=0.4]{$\sbmat{\hat y_1 \\ \hat u_1}$} (D2);
			\draw[arr]  (D2) |- node{$\hat u_1$} (algo);
		\end{scope}

		\begin{scope}[shift={(5.3,-4.2)}]
			\node[blk] at (0,0) (algo) {$\bmat{\hat H_1\\I}$};
			\node[blk] at (1.5,0) (Mi) {$\hat M^{-1}$};
			\node[blk] at (-1.5,0) (D1) {$\bmat{0 & I}$};
			\node[blk] at (3,0) (D2)  {$\bmat{I & 0}$};
			\node[blk] at (-3,0) (M) {$\hat M$};
			\coordinate (M1) at ($(M.west)+(0,0.2)$);
			\coordinate (M2) at ($(M.west)+(0,-0.2)$);
			\node[blk] (oracle) at (0,-1.5) {$\Phi_2$};
			\draw[arr] (M) -- (D1);
			\draw[arr] (D1) -- (algo);
			\draw[arr] (algo) -- (Mi);
			\draw[arr] (Mi) -- (D2);
			\coordinate (D2x) at ($(D2.east) + (0.6,0)$);
			\draw[arr] (D2) -- (D2x) |- node[pos=0.75]{$\hat y_2$}(oracle);
			\draw[arr] (oracle) -| node[pos=0.25]{$\hat u_2$} ($(M2)+(-0.6,0)$) -- (M2);
			\draw[arr] (D2) -- (D2x) -- +(0,0.9) -| node[pos=0.25,swap] {$\hat y_2$} ($(M1)+(-0.6,0)$) -- (M1);
			\coordinate (c1) at ($(M.south west) + (-1,-0.4)$);
			\coordinate (c2) at ($(D2.north east) + (1,1.1)$);
			\coordinate (c3) at (c2 |- c1);
			\coordinate (cm) at ($(c3)!0.5!(c2)$);
			\draw[densely dashed] (c1) rectangle (c2);
			\node[right] at (cm) {$\hat H_2$};
		\end{scope}

	\end{tikzpicture}
	\caption{Equivalent block diagrams representing LFT equivalence. Starting from the top left, we augment the algorithm and oracle, we transform the oracle using the linear equivalence $\Phi_1 \Msim \Phi_2$, and finally we isolate the equivalent $\hat H_2$ in feedback with $\Phi_2$.}
	\label{fig_LFT_transformation}
\end{figure}

The dashed box in \cref{fig_LFT_transformation} represents the equivalent $\hat H_2$. Based on the block diagram, we obtain the following algebraic relationships:
\[
\hat y_2 = \bmat{I & 0} \hat M^{-1} \bmat{\hat H_1 \\ I} \bmat{0 & I} \hat M \bmat{\hat y_2 \\ \hat u_2}
\qquad\text{and}\qquad
\hat y_2 = \hat H_2 \hat u_2.
\]
These equations can be resolved in various ways. Most relevant for our purpose, we eliminate $\hat y_2$ and seek an identity that holds for all $\hat u_2$, which leads to:
\[
\hat H_2 = \bmat{I & 0} \hat M^{-1} \bmat{\hat H_1 \\ I} \bmat{0 & I} \hat M \bmat{\hat H_2 \\ I}.
\]
This expression can be further simplified to obtain the relationships in \cref{thm:main_LFT_result,cor:main_LFT_result}.

The special cases of \cref{prop:LFT_properties} lead to simple expressions for LFT equivalence between algorithms.

\medskip
\begin{corollary}[commutation]\label{cor:commutation}
	Suppose $\phi(\hat C \hat y) = \hat C\phi(\hat y)$ for all $\hat y$.
	Then ${\hat H_1 \algcomp \phi \Msim \hat H_2 \algcomp \phi}$ with $\hat M = \sbmat{\hat C & 0 \\ 0 & \hat C}$. Consequently, $\hat H_1 = \hat C \hat H_2 \hat C^{-1}$. If we let $\hat C = \hat \Delta_m$ (multi-shift), then we see that shift equivalence is a special case of LFT equivalence.
\end{corollary}
\medskip

\begin{corollary}[equivariance]\label{cor:equivariance}
	Suppose $\phi_1(\hat A \hat y) = \hat B\phi_2(\hat y)$ for all $\hat y$.
	Then \\
	$\hat H_1 \algcomp \phi_1 \Msim \hat H_2 \algcomp \phi_2$ with $\hat M = \sbmat{\hat A & 0 \\ 0 & \hat B}$. Consequently, $\hat H_1 = \hat A \hat H_2 \hat B^{-1}$.
\end{corollary}
\medskip

We are also interested in the special case of algorithms $\hat H_1 \algcomp(\Psi,\phi_1)$ and $\hat H_2\algcomp(\Psi,\phi_2)$, where $\Psi$ is some set of oracles common to both algorithms, and $\phi_1\Msim\phi_2$ with $\hat M = \sbmat{\hat P & \hat Q \\ \hat R & \hat S}$ In this case, by concatenation (\cref{prop:LFT_properties}), we have $(\Psi,\phi_1)\Mpsim(\Psi,\phi_2)$ with $\hat M' = \sbmat{I & 0 & 0 & 0 \\ 0 & \hat P & 0 &\hat Q \\ 0 & 0 & I & 0 \\ 0 & \hat R & 0 & \hat S}$. Therefore, we immetiately obtain the following corollary.

\medskip
\begin{corollary}[LFT equivalence with common oracles]\label{cor:scalar_LFT2}
Suppose $\phi_1 \Msim \phi_2$ with $\hat M = \sbmat{\hat P & \hat Q \\ \hat R & \hat S}$.
Let $\Psi$ be another oracle. Then $\hat H_1 \algcomp (\Psi,\phi_1) \Mpsim \hat H_2 \algcomp (\Psi,\phi_2)$ with $\hat M'$ defined above, if (see \cref{cor:main_LFT_result}):
\(
\hat H_1\left( \bmat{0 & 0 \\ 0 & \hat R}\hat H_2 + \bmat{I & 0 \\ 0 & \hat S}\right) = \bmat{I & 0 \\ 0 & \hat P}\hat H_2 + \bmat{0 & 0 \\ 0 & \hat Q}.
\)
\end{corollary}
\medskip

\titleparagraph{Efficient determination of LFT equivalence}

To determine whether $\hat H_1 \algcomp \Phi_1 \sim \hat H_2 \algcomp \Phi_2$, we must first establish how $\Phi_1$ and $\Phi_2$ are related. If there exists some $\hat M$ such that $\Phi_1 \Msim \Phi_2$, then we can apply \cref{thm:main_LFT_result}, and we have $\hat H_1 \algcomp \Phi_1 \sim \hat H_2 \algcomp \Phi_2$ if $\bmat{I & -\hat H_1} \hat M \bmat{\hat H_2 \\ I} = 0$. For an example of how this result can be used in practice, see the end of \cref{ex:LFT_equiv}.

If there are many different $\hat M$ matrices that work, say $\Phi_1 \Msim\Phi_2$ for all $\hat M \in \mathcal{M}$, then it follows from \cref{def:lin_related_oracles} that $\mathcal{M}$ is a multiplicative group. Determining equivalence amounts to checking feasiblity of the problem $\bmat{I & -\hat H_1} \hat M \bmat{\hat H_2 \\ I} = 0$ with $\hat M \in \mathcal{M}$. We saw at the end of \cref{shift-equ} how to solve this problem for the special case of shift-equivalence ($\mathcal{M}$ is the set of multi-shifts). However, we suspect that solving this problem for different $\mathcal{M}$ would require a case-by-case analysis.

\subsection{Proxes, subdifferentials, and their conjugates}\label{sec:prox_subdiff}

We are now ready to return to the motivating examples of \crefrange{algo11}{algo12b}. The oracles $\{\partial f, \partial f^*, \prox_{tf}, \prox_{\frac{1}{t}f^*}\}$ are linearly equivalent to one another. Using the identities at the beginning of \cref{lft-equ}, these relationships can be derived as in \cref{eq:example_prox_eqn}. The associated matrices $\hat M$ corresponding to \cref{def:lin_related_oracles} are given in \cref{table:LFTproxes}.

% \vspace{-5mm}
\begin{figure}[htb!]
	\centering
\begin{equation*}
\def\dy{4mm}
\begin{array}{r|cccc}
	\def\dw{1.2cm}
	\def\dh{0.7cm}
	\begin{tikzpicture}[baseline={(b)}]
		\node[minimum height=\dh, minimum width=\dw] (box) {};
  		\path (box.center) -- (box.south) coordinate[pos=0.25] (b);
		\node[anchor=south west, inner sep=0] at (box.north west) {\phantom{A}};
		\draw ($(box.north west) + (0.3,0)$) -- ($(box.south east) + (-0.3,0)$);
		\node[anchor=north east] at (box.north east) {$\phi_2$};
		\node[anchor=south west] at (box.south west) {$\phi_1$};
	\end{tikzpicture} & \partial f & \partial f^* & \prox_{tf} & \prox_{\frac{1}{t}f^*} \\ \midrule
\partial f\;\;
	& \bmat{1 & 0 \\ 0 & 1}
	& \bmat{0 & 1 \\ 1 & 0}
	& \bmat{0 & 1 \\ \frac{1}{t} & -\frac{1}{t}}
	& \bmat{t & -t \\ 0 & 1} \\[\dy]
\partial f^*\;\;
	& \bmat{0 & 1 \\ 1 & 0}
	& \bmat{1 & 0 \\ 0 & 1}
	& \bmat{\frac{1}{t} & -\frac{1}{t} \\ 0 & 1}
	& \bmat{0 & 1 \\ t & -t} \\[\dy]
\prox_{tf}\;\;
	& \bmat{1 & t \\ 1 & 0}
	& \bmat{t & 1 \\ 0 & 1}
	& \bmat{1 & 0 \\ 0 & 1}
	& \bmat{t & 0 \\ t & -t} \\[\dy]
\prox_{\frac{1}{t}f^*}\;\;
	& \bmat{\frac{1}{t} & 1 \\ 0 & 1}
	& \bmat{1 & \frac{1}{t} \\ 1 & 0}
	& \bmat{\frac{1}{t} & 0 \\ \frac{1}{t} & -\frac{1}{t}} 
	& \bmat{1 & 0 \\ 0 & 1} \\
\end{array}
\end{equation*}
\caption{Matrices $\hat M$ conforming to \cref{def:lin_related_oracles} for all possible linear equivalences between the oracles $\{\partial f, \partial f^*, \prox_{tf}, \prox_{\frac{1}{t}f^*}\}$.}\label{table:LFTproxes}
\end{figure}

Note that the diagonal entries of \cref{table:LFTproxes} are \emph{identity LFT matrices} (see \cref{prop:LFT_properties}). Applying \cref{cor:scalar_LFT2} to the matrices in \cref{table:LFTproxes}, we can obtain a set of algorithms equivalent when we swap one oracle for another.

\medskip
\begin{corollary}[LFT equivalence for prox] \label{thm:conj_equivalence}
	Suppose $\hat H$ is an algorithm that uses oracles partitioned as $(\Psi, \prox_{tg})$. Then, the following transfer functions correspond to LFT-equivalent algorithms.
	\begin{align*}
		\bmat{\hat H_{11} & \hat H_{12} \\ \hat H_{21} & \hat H_{22}} &
		\algcomp (\Psi, \prox_{tg} ),
		\\
		\bmat{\hat H_{11}+\hat H_{12}(I-\hat H_{22})^{-1}\hat H_{21} & -t \hat H_{12}(I-\hat H_{22})^{-1} \\
		\frac{1}{t}(I-\hat H_{22})^{-1}\hat H_{21} & -\hat H_{22}(I-\hat H_{22})^{-1}} &
		\algcomp (\Psi, \prox_{\frac{1}{t}g^*} ),
		\\
		\bmat{\hat H_{11}+\hat H_{12}(I-\hat H_{22})^{-1}\hat H_{21} & -t \hat H_{12}(I-\hat H_{22})^{-1} \\
		(I-\hat H_{22})^{-1}\hat H_{21} & -t(I-\hat H_{22})^{-1}} &
		\algcomp (\Psi, \partial g ),
		\\
		\bmat{\hat H_{11} & \hat H_{12} \\ \frac{1}{t}\hat H_{21} & -\frac{1}{t}(I-\hat H_{22})} &
		\algcomp (\Psi, \partial g^* ).
	\end{align*}
\end{corollary}

\medskip
\begin{corollary}[LFT equivalence for subdifferentials]\label{thm:conj_equivalence2}
	Suppose $\hat H$ is an algorithm that uses oracles partitioned as $(\Psi,\partial g)$. Then, the following transfer functions correspond to LFT-equivalent algorithms.
	\begin{align*}
	\bmat{\hat H_{11} & \hat H_{12} \\ \hat H_{21} & \hat H_{22}} &\algcomp (\Psi, \partial g ) \\
	\bmat{\hat H_{11} - \hat H_{12} \hat H_{22}^{-1} \hat H_{21} &
	\hat H_{12}\hat H_{22}^{-1} \\
	-\hat H_{22}^{-1}\hat H_{21} &
	\hat H_{22}^{-1} } &\algcomp (\Psi,\partial g^*) \\
	\bmat{\hat H_{11} - \hat H_{12} \hat H_{22}^{-1} \hat H_{21} &
	\hat H_{12}\hat H_{22}^{-1} \\
	-t \hat H_{22}^{-1}\hat H_{21} &
	I + t \hat H_{22}^{-1} } &\algcomp (\Psi,\prox_{tg}) \\
	\bmat{\hat H_{11} & \hat H_{12} \\ \frac{1}{t}\hat H_{21} & \frac{1}{t}\hat H_{22}+I} &\algcomp (\Psi,\prox_{\frac{1}{t}g^*})
\end{align*}
\end{corollary}
\medskip

\begin{remark}
	If there is no $\Psi$ in \cref{thm:conj_equivalence,thm:conj_equivalence2} (the prox or subdifferential is the only oracle), then we can extract the $(2,2)$ blocks of all submatrices and we obtain LFT-equivalence among:
	\[
	\hat H \algcomp \prox_{t g},\quad
	-\hat H(I-\hat H)^{-1} \algcomp \prox_{\frac{1}{t}g^*},\quad
	-t(I-\hat H)^{-1} \algcomp 	\partial g,\quad
	-\tfrac{1}{t}(I-\hat H) \algcomp \partial g^*
	\]  
	and among:\quad
	$
	\hat H \algcomp \partial g,\quad
	\hat H^{-1} \algcomp \partial g^*,\quad
	(I+t\hat H^{-1}) \algcomp \prox_{tg},\quad
	(\tfrac{1}{t}\hat H+I) \algcomp \prox_{\frac{1}{t}g^*}
	$.
\end{remark}

\medskip
\begin{remark}\label{rem:invertible_tf}
	In \cref{thm:conj_equivalence2}, $\hat H_{22}$ must be invertible. We may want to also ensure that $\hat H_{22}^{-1}$ is proper, as this is necessary if we want an implementable algorithm. The condition that a transfer function $\hat H$ be invertible and proper can be characterized precisely \cite[Lem.~3.15]{ZDG}; it is equivalent to requiring that $D = \lim_{z\to\infty} \hat H(z)$ is invertible. One possible state-space realization of the inverse transfer function $\hat H^{-1}$ is\looseness=-1
	\[
		\hat H^{-1} =
		\left[\begin{array}{c|c}
			A & B \\ \hline
			C & D
		\end{array}\right]^{-1}
		=
		\left[\begin{array}{c|c}
			A-BD^{-1}C & BD^{-1} \\ \hline
			-D^{-1}C & D^{-1}
		\end{array}\right].
	\]
\end{remark}

\subsection{Examples of LFT equivalence} \label{ex:LFT_equiv}

\titleparagraph{\Crefrange{algo11}{algo12b}}

We can verify equivalence of \Crefrange{algo11}{algo12b} by directly applying \cref{thm:conj_equivalence}. Specifically, substituting $\hat H = \bmat{0 & \frac{1}{z}\\-t & \frac{1}{z}}$ in \cref{thm:conj_equivalence}, we immediately obtain the transfer functions in \cref{eq:conj_tfs}.

\titleparagraph{Oracle swapping and deletion}

Although we did not cover oracle swapping or deletion, both of these notions are trivial to check in our framework. In fact, they are special cases of LFT equivalence.

Most splitting methods are not symmetric with respect to oracle swapping since the different oracles usually have different properties that the algorithm is trying to exploit. Nevertheless, one might be interested in, \eg, Davis--Yin splitting where $f$ and $g$ are swapped (see \cref{algoTOS}). This is called \emph{dual Davis--Yin} in \cite{jiang2023bregman}.
Given $\hat H \algcomp \Psi$, if $\Psi$ is a permutation of the oracles $\Phi$, then we can write $\Psi(P x) = P \Phi(x)$, where $P$ is a permutation matrix. By \cref{cor:equivariance}, we have
$\hat H \algcomp \Psi = P^\top \hat H P \algcomp \Phi$,
\ie, permute the corresponding rows and columns of the transfer matrix. We exploit this fact later in this section when we show that PD3O is LFT-equivalent to Davis--Yin splitting. 

\titleparagraph{DR and Chambolle--Pock}
We can use LFT equivalence to show the relation between Douglas--Rachford (DR), \cref{algo5}, and the primal-dual optimization method proposed
by Chambolle and Pock (\cref{algo14}~\cite{chambolle2011first,o2018equivalence}).

\vspace{-1em}
\begin{algorithm}[H]
	\caption{Chambolle--Pock}
	\captionsetup{font=scriptsize}
	\centering
	\label{algo14}
	\scriptsize
	\begin{algorithmic}
		\For{$k=0, 1, 2,\ldots$}
		\State{$x^{k+1}_1 = \prox_{\tau f}(x^k_1 - \tau M^\top  x^k_2)$}
		\State{$x^{k+1}_2 = \prox_{\sigma g^*}(x^k_2 + \sigma M (2x^{k+1}_1 - x^k_1))$}
		\EndFor
	\end{algorithmic}
\end{algorithm}
\vspace{-1em}

Comparing \cref{algo14} and \cref{algo5}, we should first set $\tau =\sigma=1$ so that the oracles correspond properly. Now, computing transfer functions, we have:
\begin{align*}
\text{\cref{algo5}:}&\; \bmat{
	\frac{-1}{z-1} & \frac{1}{z-1} \\[1mm]
	\frac{2 z-1}{z-1} & \frac{-1}{z-1}}
	\algcomp (\prox_f,\prox_g)
	\quad\text{and}
	\\
\text{\cref{algo14}:}&\;	\bmat{
	\frac{1}{z} & -\frac{1}{z}M^\top  \\[1mm]
 \tfrac{2z-1}{z}M  & \frac{1}{z}}
 \algcomp (\prox_f,\prox_{g^*}).
\end{align*}
Applying \cref{thm:conj_equivalence}, we will have LFT-equivalence between these algorithms if
\[
	\bmat{
	\frac{1}{z} & -\frac{1}{z}M^\top  \\[1mm]
 	\tfrac{2z-1}{z}M  & \frac{1}{z}}
=\bmat{
	\frac{1}{z} & -\frac{1}{z} \\[1mm]
 \tfrac{2z-1}{z}  & \frac{1}{z}}
\]
Therefore, \cref{algo5,algo14} are LFT-equivalent if $M=I$.

\titleparagraph{More three-operator splitting}

An algorithm that has recently attracted considerable attention is the three-operator splitting algorithm of Davis and Yin \cite{DavisYin}. This algorithm solves the problem
\begin{equation*}
	\ba{ll}
	\mbox{minimize} & f(x) + g(x) + h(x)
	\ea
\end{equation*}
using the oracles $\prox_f$, $\prox_g$, and $\nabla h$.  The algorithm and its transfer function are given as follows.

\vspace{-1em}
\begin{minipage}{0.5\linewidth}
	\hfil\noindent
	\begin{algorithm}[H]
		\centering
		\captionsetup{font=scriptsize}
		\caption{Davis--Yin three-operator splitting}
		\label{algoTOS}
		\scriptsize
		\begin{algorithmic}
			\For{$k=0, 1, 2,\ldots$}
			\State{$z^{k} = \prox_{tf}(y^k)$}
			\State{$x^{k} = \prox_{tg}(2z^k-y^k-t \nabla h(z^k))$}
			\State{$y^{k+1} = y^k-z^k+x^k$}
			\EndFor
		\end{algorithmic}
	\end{algorithm}
\end{minipage}
\hfil
\begin{minipage}{0.48\linewidth}
	\[
	\hat H_{\ref{algoTOS}}(z) =
	\bmat{\frac{-1}{z-1} & \frac{1}{z-1} & 0 \\
	 \frac{2 z-1}{z-1} & \frac{-1}{z-1} & -t \\
	 1 & 0 & 0 }
	\]
\end{minipage}
	\vspace{1em}

\noindent Suppose we wanted to design an equivalent algorithm that used the oracles $(\prox_{tf},\prox_{g^*}, \nabla h)$ instead. We proceed in steps:
	\begin{align*}
&		\bmat{\frac{-1}{z-1} & \frac{1}{z-1} & 0 \\
\frac{2 z-1}{z-1} & \frac{-1}{z-1} & -t \\
1 & 0 & 0 }
	\algcomp (\prox_{tf},\prox_{tg}, \nabla h)\\
&
\sim \bmat{\frac{-1}{z-1} & 0 & \frac{1}{z-1} \\
1 & 0 & 0 \\
\frac{2 z-1}{z-1} & -t & \frac{-1}{z-1} }
\algcomp (\prox_{tf}, \nabla h,\prox_{tg}) && \pmat{\text{swap last two rows}\\\text{and columns}}\\
&
\sim \bmat{\frac{1}{z} & -\frac{t}{z} & -\frac{t}{z} \\
 1 & 0 & 0 \\
 \frac{2 z-1}{tz} & -\frac{(z-1)}{z} & \frac{1}{z} }
\algcomp (\prox_{tf}, \nabla h,\prox_{\frac{1}{t}g^*}) &&
\pmat{\text{apply \cref{thm:conj_equivalence}}\\
\text{with $\Phi = (\prox_{tf},\nabla h)$}}\\
&
\sim \bmat{\frac{1}{z}  & -\frac{t}{z} & -\frac{t}{z}\\
\frac{2 z-1}{tz}  & \frac{1}{z} & -\frac{(z-1)}{z}\\
1 & 0 & 0 }
\algcomp (\prox_{tf},\prox_{\frac{1}{t}g^*}, \nabla h) && \pmat{\text{swap last two rows}\\\text{and columns}}
\end{align*}
Now, compare this algorithm to PD3O (\cref{algoPD3O}), which is
\[
	\bmat{\frac{1}{z} & \frac{-\tau A^\top  }{z} & \frac{-\tau }{z} \\
	\frac{\sigma  (2 z-1)A}{z} & \frac{1}{z} & \frac{-\sigma  \tau  (z-1) A}{z} \\
	1 & 0 & 0 } \algcomp (\prox_{tf},\prox_{\sigma g^*}, \nabla h)
\]
We can see that the algorithms are LFT-equivalent upon setting $A=I$, $\tau=t$, $\sigma=\frac{1}{t}$. Although this result is known \cite{PD3O,jiang2023bregman}, the benefit of systematizing algorithm equivalence is that these sorts of equivalences can be determined straightforwardly. We can directly verify the equivalence above by applying \cref{thm:main_LFT_result} with $\hat M = \sbmat{t & 0 \\ t & -t}$ taken from \cref{table:LFTproxes} and applied to the middle oracle to transform $\prox_{t g}$ to $\prox_{\frac{1}{t} g^*}$:
\[
\bmat{I & -\hat H_{\ref{algoTOS}}} \hat M \bmat{\hat H_{\ref{algoPD3O}} \\ I}
= \bmat{1 & 0 & 0 & \frac{1}{z-1} & \frac{-1}{z-1} & 0 \\
        0 & 1 & 0 & \frac{1-2z}{z-1} & \frac{1}{z-1} & t \\
		0 & 0 & 1 & -1 & 0 & 0}
\bmat{ 1 & 0 & 0 & 0 &  0 & 0 \\
	   0 & t & 0 & 0 &  0 & 0 \\
	   0 & 0 & 1 & 0 &  0 & 0 \\
	   0 & 0 & 0 & 1 &  0 & 0 \\
	   0 & t & 0 & 0 & -t & 0 \\
	   0 & 0 & 0 & 0 &  0 & 1}
\bmat{ \frac{1}{z} & \frac{-t}{z} & \frac{-t}{z} \\
       \frac{2z-1}{t z} & \frac{1}{z} & \frac{1-z}{z} \\
	   1 & 0 & 0 \\
	   1 & 0 & 0 \\
	   0 & 1 & 0 \\
	   0 & 0 & 1 } = 0.
\]

%%%%%%%%%%%%%%%%%%%%%%%%%%%%%%%%%%%%%%%%%%%%%%%%%%%%%%%%%%%%%%%%%%%%%%%%%%%%%%%%%%%%%%%%%%%
\section{Discussion}\label{discussion}

\titleparagraph{One algorithm, many interpretations and implementations}
Is it useful to have many different forms of an algorithm, if
all the forms are LFT-equivalent?
Yes: different rewritings of one algorithm
often yield different (``physical'') intuition.
For example,
\cref{algo_i1} uses the current
loss function for extrapolation~\cite{vasilyev2010extragradient};
while \cref{algo_i2} seems to extrapolate
from the previous loss function~\cite{censor2011subgradient}.
The distributed algorithms \cref{algoNIDS,algoExDIFF}, although equivalent,
were developed using very different intuition.
The former used a \emph{gradient differencing scheme} \cite{NIDS}
whereas the latter used an \emph{adapt-correct-combine} approach \cite{ExDIFF}.

Equivalent algorithms can differ in memory usage, computational efficiency,
or numerical stability.
For example, implementations of
\cref{algo_i3,algo_i4} lead to different memory usage~\cite{daskalakis2018training, malitsky2015projected}.
At each time step $k$, \cref{algo_i3} needs to store $x_2^{k+1}$ and $F(x_2^k)$,
but \cref{algo_i4} only needs to store $x_1^{k}$ in memory.
These different rewritings also naturally yield different generalizations,
for example, by projecting different state variables.
Likewise, Douglas--Rachford (\cref{algo5}) only requires storing $x_3^k$ at each time step $k$, whereas simplified ADMM (\cref{algo6}) requires storing $\xi_2^k$ and $\xi_3^k$. This is evident from \cref{fig:algo_unroll_compare}; the dotted lines cross one arrow for Douglas--Rachford and two arrows for ADMM.

\titleparagraph{Stochastic and randomized algorithms}

Our framework applies to stochastic or randomized algorithms with almost no modifications,
simply by allowing random oracles.
For example, we can accept oracles like
random search $\min_{i=1,\dots,k} f(x,\omega_i)$,
stochastic gradient $\nabla f(x) + \omega$,
or noisy gradient $\nabla f(x+\omega)$.
The definition of oracle equivalence requires a slight modification in this setting:
for algorithms that use randomized oracles,
two algorithms are oracle-equivalent if they generate identical sequences of oracle
calls when evaluated along the same sample path.

\titleparagraph{Time-varying algorithms}

The linear time-invariant (LTI) algorithm assumption is critical, as the ability to relate the $z$-transforms of the input and output via multiplication with a transfer function ($\hat y = \hat H \hat u$) critically relies on the map $(u^0,u^1,\dots)\mapsto(y^0,y^1,\dots)$ being LTI.

Nevertheless, many of the other concepts from \cref{control} do extend to systems that are time varying.
For example, an algorithm with parameters that change on a fixed schedule but is otherwise linear, such as gradient descent with a diminishing stepsize, can be regarded as a linear time-varying (LTV) system~\cite{antsaklis2006linear}, 
and the notion of a transfer function has been generalized to LTV systems~\cite{LTV_TF}.
If, instead, the parameters change adaptively based on the other state variables, the system can be regarded as a linear parameter varying (LPV) system~\cite{LPV_book} or a switched system~\cite{sun2006switched}. Examples of such algorithms include nonlinear conjugate gradient methods and quasi-Newton methods.

\titleparagraph{Oracle structure}

We assumed throughout this paper that all oracles were \emph{nonlinear and time-invariant}.
If we weaken this assumption, and let the oracles be \emph{nonlinear and time-varying}, the notion of oracle equivalence is still meaningful: it holds if the two algorithms invoke the same sequence of oracle calls. However, shift equivalence no longer works because time-varying operators do not commute with time shifts.

If we strengthen the assumption instead, and assume the oracles are endowed with additional structure, then further equivalences are possible. Indeed, every commutation relation satisfied by the oracle leads to a new notion of equivalence!
 
For example, an oracle that is \emph{linear and time-invariant} would commute with any other LTI system (not just multi-shifts). 
As an example, consider DR (\cref{algo5}) where $f$ is known to be a quadratic function. In this case, the oracle $L = \prox_f$ is \emph{linear} and therefore commutes with \emph{any LTI system}. 
For example, it commutes with the dynamical system:
\[
\left\{\begin{aligned}
x^{k+1} &= x^k - \alpha u^k \\
y^k &= x^k
\end{aligned}\right\}
\quad=\quad
\left[\begin{array}{c|c}
	1 & -\alpha \\ \hline
	1 & 0
\end{array}\right]
\quad=\quad
\frac{1}{z-\alpha}
\]
Assuming $x^0=0$, this dynamical system maps $(u^0,u^1,\dots)\mapsto(y^0,y^1,\dots)$, with
\begin{equation}\label{filter}
y^k = u^k+\alpha u^{k-1} + \alpha^2 u^{k-2} + \cdots + \alpha^k u^0
\qquad\text{for }k=0,1,\dots
\end{equation}
This transformation clearly commutes with a linear oracle $L$, because left-multiplying each $u^k$ by $L$ and then applying the transformation \eqref{filter} is the same as applying \eqref{filter} first and then left-multiplying by $L$. In other words,
\[
L \cdot \tfrac{1}{z-\alpha} \cdot \hat y = \tfrac{1}{z-\alpha} \cdot L \cdot \hat y.
\]
Since DR uses oracles $(\prox_f,\prox_g)$ and only $\prox_f$ is assumed to be linear, the special commutation relation only holds for $\prox_f$, and we may write
\[
\bmat{ \prox_f & 0 \\ 0 & \prox_g} \bmat{\tfrac{1}{z-\alpha} & 0 \\ 0 & 1} \bmat{\hat y_1 \\ \hat y_2}
= \bmat{\tfrac{1}{z-\alpha} & 0 \\ 0 & 1} \bmat{ \prox_f & 0 \\ 0 & \prox_g}\bmat{\hat y_1 \\ \hat y_2}.
\]
So when $f$ is a quadratic function, we have via \cref{cor:commutation} that $\hat H_1 \algcomp (\prox_f,\prox_g) \sim \hat H_2 \algcomp (\prox_f,\prox_g)$
if $\hat H_2 = \sbmat{\tfrac{1}{z-\alpha} & 0 \\ 0 & 1}^{-1}\hat H_1\sbmat{\tfrac{1}{z-\alpha} & 0 \\ 0 & 1}$. Letting $\hat H_1 = \hat H_{\ref{algo5}}$, we obtain the equivalent algorithm:	
\[
	\bmat{\frac{-1}{z-1} & \frac{1}{z-1} \\ \frac{2 z-1}{z-1} & \frac{-1}{z-1}}\algcomp (\prox_f,\prox_g)
	\sim
	\bmat{\frac{-1}{z-1} & \frac{z-\alpha}{z-1} \\ \frac{2 z-1}{(z-1)(z-\alpha)} & \frac{-1}{z-1}} \algcomp (\prox_f,\prox_g)
\]
One possible realization of the new algorithm is given below.
\vspace{-1em}
\begin{algorithm}[H]
	\centering
	\captionsetup{font=scriptsize}
	\caption{Quadratic-$f$ variant of DR}
	\label{algo_whoa}
	\scriptsize
	\begin{algorithmic}
		\For{$k=0, 1,\dots$}
		\State{$y^k=\prox_g(x_2^k-2x_1^k)$}
		\State{$x_1^{k+1} = \alpha x_1^k - \prox_f\bigl( \alpha x_2^k - x_2^{k+1}\bigr)$}
		\State{$x_2^{k+1} = x_2^k - x_1^k - y^k$}
		\EndFor
	\end{algorithmic}
\end{algorithm}
\vspace{-1em}
Therefore, \cref{algo_whoa,algo5} are equivalent for all $\alpha$ when $f$ is a quadratic function, but they cease to be equivalent when we remove this constraint on $f$. Specific equivalence results that require one of the oracles to be linear can be found, for example, in \cite[Theorem 4]{yan2016self}. However, the approach presented above is far more general, as it allows one to systematically derive entire families of equivalent algorithms.

Moving beyond linearity, different notions of equivalence could conceivably be developed for other classes of oracles, such as dynamic oracles (oracles with memory), or multi-dimensional oracles that have structure, such as sparsity.

\titleparagraph{Nonlinear state updates}

Our main exposition only considers algorithms defined by state-space equations:
a linear map relates $(x^k, u^k)$ to $(x^{k+1},y^k)$.
However, this assumption can be relaxed:
linear state updates is a sufficient condition, but it is not necessary. 
We only require that the map $(u^0,u^1,\dots)\mapsto(y^0,y^1,\dots)$ be LTI. 
For example,  consider \cref{algop3}, which is related to ordinary gradient descent (\cref{algo4}) 
via a nonlinear state transformation.

\vspace{-1em}
\begin{algorithm}[H]
	\centering
	\captionsetup{font=scriptsize}
	\caption{}
	\label{algop3}
	\scriptsize
	\begin{algorithmic}
		\For{$k=0,1, 2,\dots$}
		\State{$x^{k+1} = x^k \exp( - \frac{1}{5} \nabla f(\log x^k))$}
		\EndFor
	\end{algorithmic}
\end{algorithm}
\vspace{-1em}

Although the state update equations for \cref{algop3} are nonlinear, 
if we identify the oracle input $y^k = \log x^k$ and the oracle output $u^k = \nabla f (y^k)$, 
we can eliminate $x^k$ and write the algorithm as $y^{k+1} = y^k - \frac{1}{5}u^k$, 
which is a linear system with transfer function $-\tfrac{1}{5}\tfrac{1}{z-1}$.

\titleparagraph{Repeated oracles}

Some algorithms make multiple calls to the same oracle at each iteration. One such example is the extragradient method \cite{nemirovski2004prox}, given as \cref{extragradient},

\vspace{-1em}
\noindent
\hfil
\begin{minipage}{0.53\linewidth}
\begin{algorithm}[H]
	\centering
	\captionsetup{font=scriptsize}
	\caption{Extragradient method}
	\label{extragradient}
	\scriptsize
	\begin{algorithmic}
		\For{$k=0, 1, 2,\ldots$}
		\State{$w^k = \pi\bigl(x^k - \gamma \phi(x^k)\bigr)$}
		\State{$x^{k+1} = \pi\bigl(x^k - \gamma \phi(w^k)\bigr)$}
		\EndFor
	\end{algorithmic}
\end{algorithm}
\end{minipage}
\hfil	
\begin{minipage}{0.45\linewidth}
	\vspace{1em}
	\[
		\hat H_{\ref{extragradient}} =
		\bmat{  0 & \frac{1}{z} & -\gamma  & 0 \\
				0 & \frac{1}{z} & 0 & -\gamma  \\
				0 & \frac{1}{z} & 0 & 0 \\
				1 & 0 & 0 & 0 }
	\]
\end{minipage}
\hfil
\vspace{1em}

\noindent which calls oracles $\pi$ (a projection) and $\phi$ (a gradient) twice at each iteration. We can model such algorithms in our framework by treating the repeated oracles as separate oracles. Here, $\hat H_{\ref{extragradient}}$ assumes the ordering $(\pi_1,\pi_2,\phi_1,\phi_2)$, where $\pi_1$ refers to the first time the projection oracle $\pi$ is called and similarly for $\pi_2,\phi_1,\phi_2$. Checking for equivalence therefore means checking for shift equivalence \emph{and} oracle permutation among the repeated oracles.

We can also envision an algorithm that trivially iterates a simpler algorithm multiple times, such as ``double gradient descent'', given as \cref{doublegradient} below,

\vspace{-1em}
\noindent
\hfil
\begin{minipage}{0.53\linewidth}
\begin{algorithm}[H]
	\centering
	\captionsetup{font=scriptsize}
	\caption{Double gradient descent}
	\label{doublegradient}
	\scriptsize
	\begin{algorithmic}
		\For{$k=0, 1, 2,\ldots$}
		\State{$w^k = x^k - \gamma \nabla f(x^k)$}
		\State{$x^{k+1} = w^k - \gamma \nabla f(w^k)$}
		\EndFor
	\end{algorithmic}
\end{algorithm}
\end{minipage}
\hfil	
\begin{minipage}{0.45\linewidth}
	\[
		\hat H_{\ref{doublegradient}} =
		\bmat{ \frac{-\gamma }{z-1} & \frac{-\gamma }{z-1} \\
			   \frac{-\gamma  z}{z-1} & \frac{-\gamma }{z-1} }
	\]
\end{minipage}
\hfil
\vspace{1em}

\noindent which assumes an oracle $\Phi = (\nabla f, \nabla f)$. In our framework we \emph{do not} consider \cref{doublegradient} to be equivalent to ordinary gradient descent with stepsize $\gamma$, because the oracles have different sizes and there would be a type mismatch if we tried to equate the sequences of oracle calls between the two algorithms. 

\titleparagraph{Computational complexity}

Checking for equivalence in our framework is straightforward using a computer algebra system to verify relations between transfer functions, as demonstrated in \cref{control}.
Importantly, the transfer functions that arise here are not arbitrary symbolic expressions; they come from linear state-space representations, so each entry is a rational function of $z$ (see \cref{app:construction}). Thus, equivalence can be checked entrywise using exact algebra on rational functions, for example by reducing each entry to numerator/denominator form and verifying the resulting polynomial identities, or equivalently by cross-multiplying and checking that the difference is the zero polynomial.

If an algorithm uses $p$ oracles and has state dimension $n$, then checking oracle or shift equivalence is polynomial in $p$ and $n$ (see the end of \cref{shift-equ}). Checking LFT equivalence also has similar complexity (see the paragraph before \cref{sec:prox_subdiff}) provided the number of $\hat M$ matrices to check is fixed or polynomial in $p$. That being said, as far as we know, all algorithms have $n\leq 3$ and $p\leq 3$, so computational complexity is not a difficulty in practice.

\titleparagraph{Beyond optimization algorithms}

The ideas in this paper are not limited to optimization algorithms, but can be applied to any iterative algorithm involving oracle evaluations and linear updates. For example, algorithms for solving monotone inclusions, variational inequalities, fixed point problems, or equilibrium computation can all be represented as linear dynamical systems in feedback with oracles. Other examples include algorithms for numerical linear algebra (linear systems of equations, least squares, eigenvalue problems) and algorithms for solving differential equations (linear multistep methods, Runge--Kutta methods).\looseness=-1

%%%%%%%%%%%%%%%%%%%%%%%%%%%%%%%%%%%%%%%%%%%%%%%%%%%%%%%%%%%%%%%%%%%%%%%%%%%%%%%%%%%%%%%%%%%
\section{Software implementation}\label{linnaeus}

We implemented our framework as a web-based application called \lin{}\footnote{Named after Carl Linnaeus, creator of the modern system for naming organisms.}, available at \url{https://udellgroup.github.io/Linnaeus_software/}. The user input is a proposed algorithm described using natural syntax, and the output is a list of all algorithms in the library that are equivalent to the input algorithm, along with the parameter settings that make them equivalent and pointers to relevant literature. \lin{} can reproduce all equivalence results mentioned in this paper. The software is open source and we welcome contributions to the library of algorithms and equivalence results.\looseness=-1

%%%%%%%%%%%%%%%%%%%%%%%%%%%%%%%%%%%%%%%%%%%%%%%%%%%%%%%%%%%%%%%%%%%%%%%%%%%%%%%%%%%%%%%%%%%
\section{Conclusion}\label{conclusion}

Our work presents first steps towards systematizing the study of optimization algorithms. When viewed as dynamical systems and characterized in terms of their input-output maps, algorithms are distilled to their essential function: a causal map that produces the next oracle input based on past oracle outputs.

Looking forward, \emph{control theory} is well-positioned to advance the fields of algorithm discovery, analysis, and design. Control theory is concerned with the analysis and synthesis of dynamical systems with the goal of obtaining desirable overall behavior, such as stability or robustness to noise. In particular, tools from \emph{robust control} have been used to analyze and design optimization algorithms with optimized convergence rates or noise-robustness properties, for example \cite{doi:10.1137/15M1009597,michalowsky2021robust}.

\begin{appendices}



\newpage
\section{Control theory results}\label{sec:app1}

\subsection{Minimal realizations}\label{app:minimality}

We start with the important definitions of \emph{controllability}, \emph{observability}, and \emph{minimality}.

\medskip
\begin{definition}\label{def:minimality}
	Consider a realization $(A,B,C,D)$ with $A\in\R^{n\times n}$. The realization, or simply the pair $(A,B)$, is \emph{controllable} if the controllability matrix $\mathcal{C}$ has full row rank. The realization, or simply the pair $(C,A)$, is \emph{observable} if the observability matrix $\mathcal{O}$ has full column rank.
	The controllability and observability matrices are defined as:
	\[
		\mathcal{C} \defeq \bmat{B & AB & \cdots & A^{n-1}B}
		\quad\text{and}\quad
		\mathcal{O} \defeq \bmat{C \\ CA \\ \vdots \\ CA^{n-1}}.
	\]
\end{definition}

\medskip
\begin{definition}
	A realization $(A,B,C,D)$ of $\hat H(z)$ is said to be \emph{minimal} if $A$ has the smallest possible dimension. 
\end{definition}
\medskip

The notions of controllability and observability are intimately connected to the notion of a minimal realization \cite[\S3.7]{ZDG}.

\medskip
\begin{proposition}\label{prop:minimality}
	A realization is minimal if and only if it is controllable and observable. Furthermore, all minimal realizations of $\hat H(z)$ are related to one another via a suitably chosen invertible matrix $T$ and the transformation \eqref{eqp11}.
\end{proposition}
\medskip

The transfer function corresponding to a realization $(A,B,C,D)$ can be expanded into an infinite series. Its (matrix) coefficients $M_k$ are called the \emph{Markov parameters} and are defined as:
\begin{align*}
	\hat H(z) &= D + C(zI-A)^{-1} B \\
	&= D + CB z^{-1} + CAB z^{-2} + CA^2 B z^{-3} + \cdots + CA^{k-1}B z^{-k} + \cdots \\
	&\eqdef M_0 + M_1 z^{-1} + M_2 z^{-2} + M_3 z^{-3} + \cdots + M_k z^{-k} + \cdots
\end{align*}
The Markov parameters only depend on the transfer function, so all realizations of a given transfer function have the same Markov parameters.
We can arrange the Markov parameters into the semi-infinite \emph{Hankel matrix}, which factors as:
\[
\bH \defeq \bmat{ M_1 & M_2 & M_3 & \cdots \\
	   M_2 & M_3 & M_4 & \cdots \\
	   M_3 & M_4 & M_5 & \cdots \\
		\vdots & \vdots & \vdots & \ddots} =
\bmat{ CB & CAB & CA^2B & \cdots \\
	   CAB & CA^2B & CA^3B & \cdots \\
	   CA^2B & CA^3B & CA^4B & \cdots \\
		 \vdots & \vdots & \vdots & \ddots}
		= \underbrace{\bmat{C \\ CA \\ CA^2 \\ \vdots}}_{\bO} 
		\underbrace{\bmat{ B & AB & A^2 B & \cdots}}_{\bC}.
\]
In particular, this identity must hold for a minimal realization, which is controllable and observable by \cref{prop:minimality}. We immediately obtain the following result.

\medskip
\begin{proposition}\label{prop:hankel}
	The minimal realization of a system has a state dimension equal to $n = \rank(\bH)$, where $\bH$ is the Hankel matrix of the system.
\end{proposition}
\medskip

\begin{remark}\label{rem:hankel_truncation}
	When applying \cref{prop:hankel} in practice, it suffices to compute a block-$N\times N$ truncation of $\bH$, where $N$ is any upper bound on the minimal state dimension. For example, we can let $N$ be the dimension of any given realization. This works because due to the Cayley--Hamilton theorem, the ranks of $\bO$ and $\bC$ can no longer increase after the $N\textsuperscript{th}$ block-row or block-column, respectively.
\end{remark}

In \cref{app:construction}, we show how to construct a minimal realization of a given transfer function.

\subsection{Proof of \texorpdfstring{\cref{prop:tf-ss,prop:tf-ss2}}{\ref{prop:tf-ss} and \ref{prop:tf-ss2}}}\label{proof:prop12}

\begin{proof}
	Applying \cref{eq:zdomain_yu_map}, we may write the input-output maps for both algorithms as $\hat y(z) = \hat O_i(z) x^0_i + \hat H_i(z) \hat u(z)$ for $i=1,2$. If there exist $x^0_1$ and $x^0_2$ such that the input-output maps of both systems are the same, we must have $\hat O_1(z) x^0_1 = \hat O_2(z)x^0_2$ and $\hat H_1(z) = \hat H_2(z)$. This establishes necessity for \cref{prop:tf-ss,prop:tf-ss2}.

	We now prove sufficiency. If $\hat H_1(z) = \hat H_2(z)$, both systems have the same input-output map if and only if $\hat O_1(z)x_1^0 = \hat O_2(z)x_2^0$. A trivial solution is $x_1^0=0$ and $x_2^0=0$, which proves sufficiency for \cref{prop:tf-ss}. Now suppose both realizations are minimal and pick any $x_2^0$. By \cref{prop:minimality}, there exists an invertible matrix $T$ such that:
	\(
	(A_2,B_2,C_2,D_2) = (TA_1T^{-1},TB_1,C_1T^{-1}D_1)
	\).
	Therefore, we have:
	\begin{align*}
		\hat O_2(z) x_2^0 &= z C_2(zI-A_2)^{-1} x_2^0
		= z C_1T^{-1}(zI-TA_1T^{-1})^{-1} x_2^0 \\
		&= z C_1 (zI-A_1)^{-1} T^{-1} x_2^0
		= \hat O_1(z) T^{-1} x_2^0.
	\end{align*}
	Setting $x_1^0 = T^{-1} x_2^0$ leads us to $\hat O_1(z)x_1^0 = \hat O_2(z)x_2^0$, as required. To prove uniqueness, suppose $\tilde x_1^0 \neq x_1^0$ is a different solution, so that $\hat O_1(z)x_1^0 = \hat O_1(z)\tilde x_1^0$. In other words, $\hat O_1(z) v = 0$ for some $v \defeq \tilde x_1^0-x_1^0\neq 0$. Then, we have:
	\[
	\hat O_1(z) v
	= z C_1(zI-A_1)^{-1}v
	= \left( C_1 + C_1 A_1 z^{-1} + C_1 A_1^2 z^{-2} + \cdots \right) v
	= 0.
	\]
	We conclude that $C_1 A_1^k v = 0$ for $k=0,1,\dots$, and therefore $\mathcal{O}_1v=0$, where $\mathcal{O}_1$ is the observability matrix (\cref{def:minimality}). Minimality of $(A_1,B_1,C_1,D_1)$ implies observability by \cref{prop:minimality}, and therefore $\mathcal{O}_1$ has full column rank and $v=0$, a contradiction. This establishes sufficiency of \cref{prop:tf-ss2} and completes the proof. 
\end{proof}

\newpage
\subsection{From transfer functions to algorithms}\label{app:construction}

We begin by summarizing some key properties of transfer functions.

\medskip
\begin{proposition}\label{prop:property_of_tf}
	Consider a state-space system $(A,B,C,D)$ and its associated transfer function $\hat H(z) = D+C(zI-A)^{-1}B$.
	\begin{enumerate}
		\item $\hat H$ is \emph{rational}, which means each entry $\hat H_{ij}(z)$ can be expressed as a ratio of polynomials $\hat H_{ij}(z) = \frac{p_{ij}(z)}{q_{ij}(z)}$ where $p_{ij}(z)$ and $q_{ij}(z)$ have no common factors.
		\item $\hat H$ is \emph{proper}, which means that $\deg(p_{ij}) \leq \deg(q_{ij})$ for all $i,j$.
		\item $\hat H$ is \emph{strictly proper}, meaning $\deg(p_{ij}) < \deg(q_{ij})$ for all $i,j$, if and only if $D = 0$.
	\end{enumerate}
\end{proposition}

\cref{prop:property_of_tf} follows immediately from the formula $\hat H(z) = D+C(zI-A)^{-1}B$. 

\titleparagraph{From transfer functions to minimal realizations}

Given that all state-space realizations yield proper transfer functions, we now show how to construct a minimal realization given an arbitrary $p\times m$ proper rational transfer function $\hat H(z)$.\footnote{When representing algorithms, we typically have $p=m$ (square $\hat H$), since there are as many oracle inputs as oracle outputs. In general, we can find state-space realizations for non-square transfer functions.}
There are many efficient methods for constructing realizations from transfer functions \cite[\S5.4]{antsaklis2006linear}. We now present one such method, the Ho--Kalman algorithm \cite{hokalman}, which is direct, efficient, and numerically stable.

The Ho--Kalman algorithm is based on the Hankel matrix and leverages \cref{prop:hankel,rem:hankel_truncation}. Let $N$ be an upper bound on the number of states of the minimal realization. One way to obtain such a bound is to let $N$ be the degree of the denominator polynomial of $\det \hat H(z)$. 
Now form the truncated Hankel matrix
\[
\mathcal{H}_N \defeq \bmat{ M_1 & M_2 & \cdots & M_N \\
	   M_2 & M_3 & \cdots & M_{N+1} \\
	   \vdots & \vdots & \iddots & \vdots \\
	   M_N & M_{N+1} & \cdots & M_{2N-1}}
		= \underbrace{\bmat{C \\ CA \\ \vdots \\ CA^{N-1}}}_{\mathcal{O}_N} 
		\underbrace{\bmat{ B & AB & \cdots & A^{N-1}B}}_{\mathcal{C}_N}
		= \mathcal{O}_N \mathcal{C}_N.
\]
The \emph{shifted} Hankel matrix (starts at $M_2$ instead of $M_1$) can also be factored:
\[
{\mathcal{H}}_N^+ \defeq \bmat{ M_2 & M_3 & \cdots & M_{N+1} \\
	   M_3 & M_4 & \cdots & M_{N+2} \\
	   \vdots & \vdots & \iddots & \vdots \\
	   M_{N+1} & M_{N+2} & \cdots & M_{2N}}
		= \mathcal{O}_N A \mathcal{C}_N.
\]
Now compute the compact singular value decomposition (SVD) $\mathcal{H}_N = U \Sigma V^\top$. By \cref{prop:hankel}, $\Sigma \in \R^{n\times n}$ and $n$ is the minimal state dimension.
We will choose $\mathcal{O}_N = U\Sigma^{1/2}$ and $\mathcal{C}_N = \Sigma^{1/2}V^\top$. These matrices are left- and right-invertible, respectively. Finally, define the following matrices.
\begin{itemize}
\item $D = M_0 = \lim_{z\to\infty} \hat H(z)$.
\item $C$ is the first $p$ rows of $\mathcal{O}_N = U\Sigma^{1/2}$.
\item $B$ is the first $m$ columns of $\mathcal{C}_N = \Sigma^{1/2}V^\top$.
\item $A = \mathcal{O}_N^\dagger \mathcal{H}_N^+ \mathcal{C}_N^\dagger
= \Sigma^{-1/2}U^\top \mathcal{H}_N^+ V \Sigma^{-1/2}$.
\end{itemize}
Then, $(A,B,C,D)$ is a minimal realization of $\hat H(z)$.
In practice, one typically chooses $N$ to be a loose upper bound, such as twice the degree of the determinant $\det \hat H(z)$, as this yields a more numerically stable SVD computation.

\medskip
\begin{remark}
	In the field of controls, the Ho--Kalman algorithm is often used as a method of \emph{system identification}, where the Markov parameters $M_k$ are measured in a physical system, and we seek a state-space model $(A,B,C,D)$ or transfer function model $\hat H(z)$ that fits the data \cite{hokalman}. Hankel singular values also show up in \emph{model reduction}, when we seek simpler approximate models (with fewer states). One way is to truncate the smallest Hankel singular values \cite[\S7]{ZDG}. This is akin to finding a low-rank approximation of a real matrix by truncating the smallest singular values.
\end{remark}

\titleparagraph{From realizations to state update equations}

As described in \cref{preliminary}, state-space realization can directly be converted back to step-by-step implementations of the form of \cref{algo_generic_iterative}. These implementations will be explicit if $D$ can be permuted into a strictly lower-triangular matrix, and implicit otherwise.

\section{Equivalence proofs}\label{app:equivalence}

\subsection{Proof of \cref{lem:shift_equiv}}\label{app:equivalence1}
	
	We verify each property separately.
	
	\smallskip
	\noindent\emph{Reflexivity}: Let $\hat \Delta_m = I$. For any $\hat H$, we trivially have $\hat H = \hat \Delta_m \hat H \hat \Delta_m^{-1}$. Therefore $\hat H \sim \hat H$, establishing reflexivity.
	
	\smallskip
	\noindent\emph{Symmetry:} Suppose $\hat H_1 \sim \hat H_2$. Therefore,
	\begin{align*}
	\hat H_1 = \hat \Delta_m \hat H_2 \hat \Delta_m^{-1}
	\;\implies\;
	\hat H_2  = \hat \Delta_m^{-1} \hat H_1 \hat \Delta_m
	= \bigl( \hat \Delta_m^{-1} z^{-M} \bigr) \hat H_1 \bigl( \hat \Delta_m^{-1} z^{-M} \bigr)^{-1},
	\end{align*}
	where we let $M \defeq \max(m_1,\dots,m_p)$. Therefore, $\hat \Delta_m^{-1}z^{-M}$ is a valid multi-shift (all powers of $z$ are nonpositive) and $\hat H_2 \sim \hat H_1$, establishing symmetry.
	
	\smallskip
	\noindent\emph{Transitivity:} Suppose $\hat H_1 \sim \hat H_2$ and $\hat H_2 \sim \hat H_3$. Therefore there exist multi-shifts $\hat \Delta_1$ and $\hat \Delta_2$ such that $\hat H_1 = \hat \Delta_1 \hat H_2 \hat \Delta_1^{-1}$ and $\hat H_2 = \hat \Delta_2 \hat H_3 \hat \Delta_2^{-1}$. Thus, we have
	\[
	\hat H_1 = \hat \Delta_1 \hat H_2 \hat \Delta_1^{-1}
	= \hat \Delta_1 \hat \Delta_2 \hat H_3 \hat \Delta_2^{-1} \hat \Delta_1^{-1}
	= \bigl(\hat \Delta_1 \hat \Delta_2\bigr) \hat H_3 \bigl(\hat \Delta_1 \hat \Delta_2\bigr)^{-1}.
	\]
	Since $\hat \Delta_1 \hat \Delta_2$ is a valid multi-shift,  we have $\hat H_1 \sim \hat H_3$, establishing transitivity.\qed

\subsection{Proof of \cref{thm:main_LFT_result}} \label{app:equivalence2}

Suppose that $\mathcal{R}\bmat{\hat H_1 \\ I} = \hat M \mathcal{R}\bmat{\hat H_2 \\ I}$. Then for any $\hat y_2$, there exists a $\hat y_1$ such that
\(
\bmat{\hat H_1\hat y_1 \\ \hat y_1} = \hat M \bmat{\hat H_2 \hat y_2 \\ \hat y_2}
\).
Multiplying both sides on the left by $\bmat{I & -\hat H_1}$, we obtain $\bmat{I & -\hat H_1}\hat M \bmat{ \hat H_2 \\ I}\hat y_2 = 0$. This holds for all $\hat y_2$, therefore $\bmat{I & -\hat H_1}\hat M \bmat{ \hat H_2 \\ I}=0$.

\smallskip
Conversely, suppose that $\bmat{I & -\hat H_1}\hat M \bmat{ \hat H_2 \\ I}=0$. Augment the block matrices to obtain:
$\bmat{I & -\hat H_1 \\ 0 & I} \hat M \bmat{I & \hat H_2 \\ 0 & I} = \bmat{\star & 0 \\ \star & \star}$, where the $\star$'s denote unimportant blocks. The left-hand side is invertible, so the right-hand side is invertible as well. Inverting both sides, the right-hand side remains block-lower triangular, and we obtain
$\bmat{I & -\hat H_2 \\ 0 & I} \hat M^{-1} \bmat{I & \hat H_1 \\ 0 & I} = \bmat{\star & 0 \\ \star & \star}$, where the $\star$'s indicate different blocks from before. Extracting the $(1,2)$ block, we obtain $\bmat{I & -\hat H_2}\hat M^{-1} \bmat{ \hat H_1 \\ I}=0$.

Now, pick $\bmat{\hat u \\ \hat v} \in \hat M \mathcal{R}\bmat{\hat H_2 \\ I}$, so there exists some $\hat y$ such that $\bmat{\hat u \\ \hat v} = \hat M \bmat{\hat H_2  \\ I} \hat y$. Multiplying on the left by $\bmat{I & -\hat H_1}$, we conclude that $\hat u = \hat H_1 \hat v$, which we can rewrite as $\bmat{\hat u \\ \hat v} = \bmat{\hat H_1 \\ I}\hat v \in \mathcal{R}\bmat{\hat H_1 \\ I}$. Therefore,
\(
\mathcal{R}\bmat{\hat H_1 \\ I} \supseteq \hat M \mathcal{R}\bmat{\hat H_2 \\ I}
\).
Similarly, pick $\bmat{\hat u \\ \hat v} \in \hat M^{-1} \mathcal{R}\bmat{\hat H_1 \\ I}$. Following similar steps, we obtain
\(
\mathcal{R}\bmat{\hat H_2 \\ I} \supseteq \hat M^{-1} \mathcal{R}\bmat{\hat H_1 \\ I}
\).
Combining the two inclusions above, we obtain $\mathcal{R}\bmat{\hat H_1 \\ I} = \hat M \mathcal{R}\bmat{\hat H_2 \\ I}$, as required. \qed

\end{appendices}

\bibliography{sn-bibliography}

@Book{boyd_vandenberghe_2004,
	author = {Boyd, Stephen and Vandenberghe, Lieven},
	title = {Convex optimization},
	publisher = {Cambridge University Press},
	address={Cambridge, United Kingdom},
	year= {2004}
}

@Book{fenchel1953convex,
	title={Convex cones, sets and functions, mimeographed notes},
	author={Fenchel, W},
	publisher = {Princeton University Press},
	address = {Princeton, NJ},
	year={1953}
}

@Article{moreau:hal-01867187,
	title = {D{\'e}composition orthogonale d'un espace hilbertien selon deux c{\^o}nes mutuellement polaires},
	author = {Moreau, Jean Jacques},
	journal = {Comptes rendus hebdomadaires des s{\'e}ances de l'Acad{\'e}mie des sciences},
	publisher = {Elsevier},
	volume = {255},
	pages = {238-240},
	year = {1962}
}

@Article{eckstein1992douglas,
	title = {On the {D}ouglas-{R}achford splitting method and the proximal point algorithm for maximal monotone operators},
	author = {Eckstein, Jonathan and Bertsekas, Dimitri P},
	journal = {Mathematical Programming},
	volume = {55},
	pages = {293--318},
	year = {1992}
}

@Article{chambolle2011first,
	title = {A first-order primal-dual algorithm for convex problems with applications to imaging},
	author = {Chambolle, Antonin and Pock, Thomas},
	journal = {Journal of mathematical imaging and vision},
	volume = {40},
	pages = {120--145},
	year = {2011}
}

@Article{o2018equivalence,
	title = {On the equivalence of the primal-dual hybrid gradient method and Douglas--Rachford splitting},
	author = {O'Connor, Daniel and Vandenberghe, Lieven},
	journal = {Mathematical Programming},
	pages = {85--108},
	volume = {179},
	year = {2020}
}

@Book{antsaklis2006linear,
	title = {Linear systems},
	author = {Antsaklis, Panos J and Michel, Anthony N},
	year = {2006},
	publisher = {Birkh\"{a}user},
	address = {Boston, MA}
}

@Book{doi:10.1137/1.9781611974997,
	title = {First-order methods in optimization},
	author = {Beck, Amir},
	year = {2017},
	publisher = {SIAM},
	address = {Philadelphia, PA}
}

@Article{doi:10.1137/080716542,
	title = {A fast iterative shrinkage-thresholding algorithm for linear inverse problems},
	author = {Beck, Amir and Teboulle, Marc},
	journal = {SIAM Journal on Imaging Sciences},
	volume = {2},
	pages = {183--202},
	year = {2009}
}

@Article{MAL-016,
	title = {Distributed optimization and statistical learning via the alternating direction method of multipliers},
	author = {Stephen Boyd and Neal Parikh and Eric Chu and Borja Peleato and Jonathan Eckstein},
	volume = {3},
	journal = {Foundations and Trends in Machine Learning},
	pages = {1--122},
	year = {2011}
}

@Article{OPT-003,
	title = {Proximal algorithms},
	year = {2014},
	volume = {1},
	journal = {Foundations and Trends in Optimization},
	pages = {127--239},
	author = {Neal Parikh and Stephen Boyd}
}

@article{popov1980modification,
  title={A modification of the Arrow-Hurwicz method for search of saddle points},
  author={Popov, Leonid Denisovich},
  journal={Mathematical notes of the Academy of Sciences of the USSR},
  volume={28},
  number={5},
  pages={845--848},
  year={1980},
  publisher={Springer}
}

@inproceedings{chiang2012online,
  title={Online optimization with gradual variations},
  author={Chiang, Chao-Kai and Yang, Tianbao and Lee, Chia-Jung and Mahdavi, Mehrdad and Lu, Chi-Jen and Jin, Rong and Zhu, Shenghuo},
  booktitle={Conference on Learning Theory},
  pages={1--6},
  year={2012}
}

@InProceedings{OMD_rakhlin,
  title = 	 {Online learning with predictable sequences},
  author = 	 {Alexander Rakhlin and Karthik Sridharan},
  pages = 	 {993--1019},
  year = 	 {2013},
  volume = 	 {30},
  series = 	 {Proceedings of Machine Learning Research}
}

@inproceedings{daskalakis2018training,
title={Training {GAN}s with optimism},
author={Constantinos Daskalakis and Andrew Ilyas and Vasilis Syrgkanis and Haoyang Zeng},
booktitle={International Conference on Learning Representations},
year={2018}
}

@article{malitsky2015projected,
  title={Projected reflected gradient methods for monotone variational inequalities},
  author={Malitsky, Yu},
  journal={SIAM Journal on Optimization},
  volume={25},
  number={1},
  pages={502--520},
  year={2015},
  publisher={SIAM}
}

@inproceedings{gidel2018a,
title={A variational inequality perspective on generative adversarial networks},
author={Gauthier Gidel and Hugo Berard and Gaëtan Vignoud and Pascal Vincent and Simon Lacoste-Julien},
booktitle={International Conference on Learning Representations},
year={2019}
}

@book{nesterov2018lectures,
	title={Lectures on convex optimization},
	author={Nesterov, Yurii},
	year={2018},
	publisher={Springer},
	address={Berlin}
}

@article{vasilyev2010extragradient,
	title={An extragradient method for finding the saddle point in an optimal control problem},
	author={Vasilyev, FP and Khoroshilova, EV and Antipin, AS},
	journal={Moscow University Computational Mathematics and Cybernetics},
	volume={34},
	number={3},
	pages={113--118},
	year={2010},
	publisher={Springer}
}

@article{censor2011subgradient,
	title={The subgradient extragradient method for solving variational inequalities in Hilbert space},
	author={Censor, Yair and Gibali, Aviv and Reich, Simeon},
	journal={Journal of Optimization Theory and Applications},
	volume={148},
	number={2},
	pages={318--335},
	year={2011},
	publisher={Springer}
}

@article{doi:10.1137/15M1009597,
	author = {Lessard, Laurent and Recht, Benjamin and Packard, Andrew},
	title = {Analysis and design of optimization algorithms via integral quadratic constraints},
	journal = {SIAM Journal on Optimization},
	volume = {26},
	number = {1},
	pages = {57-95},
	year = {2016}
}

@article{hu2020analysis,
	title={Analysis of biased stochastic gradient descent using sequential semidefinite programs},
	author={Hu, Bin and Seiler, Peter and Lessard, Laurent},
	journal={Mathematical Programming},
	pages={1--26},
	year={2020},
	publisher={Springer}
}

@article{MAL-050,
	year = {2015},
	volume = {8},
	journal = {Foundations and Trends® in Machine Learning},
	title = {Convex optimization: algorithms and complexity},
	number = {3-4},
	pages = {231-357},
	author = {Sébastien Bubeck}
}

@book{williams2007linear,
	title={Linear state-space control systems},
	author={Williams, Robert L and Lawrence, Douglas A},
	year={2007},
	publisher={John Wiley \& Sons},
	address={Hoboken, NJ}
}

@article{lions1979splitting,
	title={Splitting algorithms for the sum of two nonlinear operators},
	author={Lions, Pierre-Louis and Mercier, Bertrand},
	journal={SIAM Journal on Numerical Analysis},
	volume={16},
	number={6},
	pages={964--979},
	year={1979},
	publisher={SIAM}
}

@article{LTV_TF,
author = {Kamen, E. W. and Khargonekar, P. P. and Poolla, K. R.},
title = {A transfer-function approach to linear time-varying discrete-time systems},
journal = {SIAM Journal on Control and Optimization},
volume = {23},
number = {4},
pages = {550-565},
year = {1985},
doi = {10.1137/0323035}
}

@book{LPV_book,
  title={Control of linear parameter varying systems with applications},
  author={Mohammadpour, J. and Scherer, C.W.},
  year={2012},
  publisher={Springer},
  address={New York}
}

@article{michalowsky2021robust,
  title={Robust and structure exploiting optimisation algorithms: an integral quadratic constraint approach},
  author={Michalowsky, Simon and Scherer, Carsten and Ebenbauer, Christian},
  journal={International Journal of Control},
  volume={94},
  number={11},
  pages={2956--2979},
  year={2021},
  publisher={Taylor \& Francis}
}

@book{ryuyinconvex,
  title={Large-scale convex optimization: algorithms \& analyses via monotone operators},
  author={Ryu, Ernest K and Yin, Wotao},
  year={2022},
  publisher={Cambridge University Press},
  address={Cambridge, United Kingdom}
}

@book{sun2006switched,
  title={Switched linear systems: control and design},
  author={Sun, Zhendong},
  year={2006},
  publisher={Springer},
  address={London, United Kingdom}
}

@article{ryu2016primer,
  title={Primer on monotone operator methods},
  author={Ryu, Ernest K and Boyd, Stephen},
  journal={Appl. Comput. Math},
  volume={15},
  number={1},
  pages={3--43},
  year={2016}
}

@misc{cvx,
	author       = {Michael Grant and Stephen Boyd},
	title        = {{CVX}: Matlab Software for Disciplined Convex Programming, version 2.1},
	howpublished = {\url{http://cvxr.com/cvx}},
	year         = 2014
}

@incollection{gb08,
	author    = {Michael Grant and Stephen Boyd},
	title     = {Graph implementations for nonsmooth convex programs},
	booktitle = {Recent Advances in Learning and Control},
	series    = {Lecture Notes in Control and Information Sciences},
	publisher = {Springer},
	address   = {London},
	pages     = {95--110},
	year      = 2008
}

@article{applicationadmm,
	title={Application of the alternating direction method of multipliers to separable convex programming problems},
	author={Masao Fukushima},
	journal={Computational Optimization and Applications},
	volume={1},
	pages={93--111},
	year={1992}
}

@article{reformulationadmm,
	title={Some reformulations and applications of the alternating direction method of multipliers},
	author={Jonathan Eckstein and Masao Fukushima},
	journal={Large Scale Optimization: State of the Art},
	pages={119--138},
	year={1993}
}

@article{phdthesis,
	author       = {Jonathan Eckstein},
	title        = {Splitting methods for monotone operators with applications to parallel optimization},
	journal={PhD thesis, MIT},
	year         = {1989}
}

@article{diamond2016cvxpy,
  title={CVXPY: A Python-embedded modeling language for convex optimization},
  author={Diamond, Steven and Boyd, Stephen},
  journal={The Journal of Machine Learning Research},
  volume={17},
  number={1},
  pages={2909--2913},
  year={2016},
  publisher={JMLR. org}
}

@inproceedings{udell2014convex,
  title={Convex optimization in Julia},
  author={Udell, Madeleine and Mohan, Karanveer and Zeng, David and Hong, Jenny and Diamond, Steven and Boyd, Stephen},
  booktitle={2014 First Workshop for High Performance Technical Computing in Dynamic Languages},
  pages={18--28},
  year={2014},
  organization={IEEE}
}

@inproceedings{shen2017disciplined,
  title={Disciplined multi-convex programming},
  author={Shen, Xinyue and Diamond, Steven and Udell, Madeleine and Gu, Yuantao and Boyd, Stephen},
  booktitle={2017 29th Chinese Control And Decision Conference (CCDC)},
  pages={895--900},
  year={2017},
  organization={IEEE}
}

@article{yan2016self,
	title={Self equivalence of the alternating direction method of multipliers},
	author={Yan, Ming and Yin, Wotao},
	journal={Splitting Methods in Communication, Imaging, Science, and Engineering},
	pages={165--194},
	year={2016},
	publisher={Springer}
}

@article{jiang2023bregman,
  title={Bregman three-operator splitting methods},
  author={Jiang, Xin and Vandenberghe, Lieven},
  journal={Journal of Optimization Theory and Applications},
  volume={196},
  number={3},
  pages={936--972},
  year={2023},
  publisher={Springer}
}

@article{DavisYin,
  title={A three-operator splitting scheme and its optimization applications},
  author={Davis, Damek and Yin, Wotao},
  journal={Set-valued and variational analysis},
  volume={25},
  pages={829--858},
  year={2017},
  publisher={Springer}
}

@article{Condat,
  title={A primal--dual splitting method for convex optimization involving Lipschitzian, proximable and linear composite terms},
  author={Condat, Laurent},
  journal={Journal of optimization theory and applications},
  volume={158},
  number={2},
  pages={460--479},
  year={2013},
  publisher={Springer}
}

@article{Vu,
  title={A splitting algorithm for dual monotone inclusions involving cocoercive operators},
  author={V{\~u}, Bằng C{\^o}ng},
  journal={Advances in Computational Mathematics},
  volume={38},
  pages={667--681},
  year={2013},
  publisher={Springer}
}

@article{PD3O,
  title={A new primal--dual algorithm for minimizing the sum of three functions with a linear operator},
  author={Yan, Ming},
  journal={Journal of Scientific Computing},
  volume={76},
  pages={1698--1717},
  year={2018},
  publisher={Springer}
}

@article{PDDY,
  title={Dualize, split, randomize: Toward fast nonsmooth optimization algorithms},
  author={Salim, Adil and Condat, Laurent and Mishchenko, Konstantin and Richt{\'a}rik, Peter},
  journal={Journal of Optimization Theory and Applications},
  volume={195},
  number={1},
  pages={102--130},
  year={2022},
  publisher={Springer}
}

@article{NIDS,
  title={A decentralized proximal-gradient method with network independent step-sizes and separated convergence rates},
  author={Li, Zhi and Shi, Wei and Yan, Ming},
  journal={IEEE Transactions on Signal Processing},
  volume={67},
  number={17},
  pages={4494--4506},
  year={2019},
  publisher={IEEE}
}

@article{ExDIFF,
  title={Exact diffusion for distributed optimization and learning---{Part I}: Algorithm development},
  author={Yuan, Kun and Ying, Bicheng and Zhao, Xiaochuan and Sayed, Ali H},
  journal={IEEE Transactions on Signal Processing},
  volume={67},
  number={3},
  pages={708--723},
  year={2018},
  publisher={IEEE}
}

@inproceedings{canform_acc,
  author = {Akhil Sundararajan and Bryan Van Scoy and Laurent Lessard},
  booktitle = {American Control Conference},
  title = {{A canonical form for first-order distributed optimization algorithms}},
  year = 2019,
  month = jul,
  pages = {4075-4080},
}

@article{DGD,
  title={Distributed subgradient methods for multi-agent optimization},
  author={Nedic, Angelia and Ozdaglar, Asuman},
  journal={IEEE Transactions on Automatic Control},
  volume={54},
  number={1},
  pages={48--61},
  year={2009},
  publisher={IEEE}
}

@article{EXTRA,
  title={{EXTRA}: An exact first-order algorithm for decentralized consensus optimization},
  author={Shi, Wei and Ling, Qing and Wu, Gang and Yin, Wotao},
  journal={SIAM Journal on Optimization},
  volume={25},
  number={2},
  pages={944--966},
  year={2015},
  publisher={SIAM}
}

@book{ZDG,
  title={Robust and optimal control},
  author={Zhou, Kemin and Doyle, John C and Glover, Keith},
  year={1996},
  publisher={Prentice-Hall},
  address={Englewood Cliffs, NJ}
}

@article{hokalman,
  title={Effective construction of linear state-variable models from input/output functions},
  author={Ho, BL and K{\'a}lm{\'a}n, Rudolf E},
  journal={at-Automatisierungstechnik},
  volume={14},
  number={1-12},
  pages={545--548},
  year={1966},
  publisher={OLDENBOURG WISSENSCHAFTSVERLAG}
}

@article{polyak1964some,
  title={Some methods of speeding up the convergence of iteration methods},
  author={Polyak, Boris T},
  journal={USSR computational mathematics and mathematical physics},
  volume={4},
  number={5},
  pages={1--17},
  year={1964},
  publisher={Elsevier}
}

@inproceedings{nesterov1983method,
  title={A method for solving the convex programming problem with convergence rate $O(1/k^2)$},
  author={Nesterov, Yurii},
  booktitle={Dokl akad nauk Sssr},
  volume={269},
  pages={543},
  year={1983}
}

@inproceedings{yan2018unified,
  title={A unified analysis of stochastic momentum methods for deep learning},
  author={Yan, Y and Yang, T and Li, Z and Lin, Q and Yang, Y},
  booktitle={IJCAI International Joint Conference on Artificial Intelligence},
  year={2018}
}

@inproceedings{ma2019qh,
  title={Quasi-hyperbolic momentum and Adam for deep learning},
  author={Jerry Ma and Denis Yarats},
  booktitle={International Conference on Learning Representations},
  year={2019}
}

@article{van2017fastest,
  title={The fastest known globally convergent first-order method for minimizing strongly convex functions},
  author={Van Scoy, Bryan and Freeman, Randy A and Lynch, Kevin M},
  journal={IEEE Control Systems Letters},
  volume={2},
  number={1},
  pages={49--54},
  year={2018},
  month=jan,
  publisher={IEEE}
}

@article{shen2023unified,
  title={A unified analysis of AdaGrad with weighted aggregation and momentum acceleration},
  author={Shen, Li and Chen, Congliang and Zou, Fangyu and Jie, Zequn and Sun, Ju and Liu, Wei},
  journal={IEEE Transactions on Neural Networks and Learning Systems},
  year={2023},
  publisher={IEEE}
}

@misc{linnaeus,
      title={An automatic system to detect equivalence between iterative algorithms}, 
      author={Shipu Zhao and Laurent Lessard and Madeleine Udell},
      year={2022},
      eprint={2105.04684},
      archivePrefix={arXiv},
      primaryClass={math.OC},
      url={https://arxiv.org/abs/2105.04684}, 
}

@article{nemirovski2004prox,
  title={Prox-method with rate of convergence $O(1/t)$ for variational inequalities with {Lipschitz} continuous monotone operators and smooth convex-concave saddle point problems},
  author={Nemirovski, Arkadi},
  journal={SIAM J. Optim.},
  volume={15},
  number={1},
  pages={229--251},
  year={2004},
  publisher={SIAM}
}

\end{document}